\documentclass[preprint,12pt]{elsarticle}
\usepackage{anyfontsize}
\usepackage{mathrsfs}
\usepackage{amsfonts}
\usepackage{arydshln}
\usepackage{amssymb}
\usepackage{amsmath}
\usepackage{amsthm}
\usepackage{multirow}
\usepackage{cases}
 \usepackage{bm}
 \usepackage{enumitem}
\usepackage[figuresright]{rotating}
\usepackage{booktabs}
\usepackage{adjustbox}
\usepackage{tikz}
\usepackage{graphicx}
\usepackage{subcaption}  
\usepackage{appendix}

\usepackage{amssymb}
\theoremstyle{plain}

\newtheorem{assumption}{Assumption}

\theoremstyle{definition}
\newtheorem{definition}{Definition}
\newtheorem{example}{Example}[section]
\newtheorem{remark}{Remark}[section]
\newtheorem{theorem}{Theorem}[section]
\newtheorem{lemma}{Lemma}[section]

\allowdisplaybreaks[2]

\begin{document}

\begin{frontmatter}

\title{\textbf{ Finite Volume Element Method on Curved-Edge Meshes}}


\author[1]{Xiaoxiao Chen}
\author[2]{Zexi Hu}
\author[2]{Zhiming Gao}
\author[1]{Junliang Lv}
\author[1]{Xiang Wang}
\author[1]{Hongtao Yang\corref{cor1}}

\address[1]{School of Mathematics, Jilin University, Changchun, China}
\address[2]{Laboratory of Computational Physics, Institute of Applied Physics and Computational Mathematics, Beijing, China}

\cortext[cor1]{\ Corresponding author} \ead{hongtao@jlu.edu.cn}

\begin{abstract}
This paper proposes and analyzes a high-order finite volume element method on curved-edge quadrilateral meshes for elliptic equations. Unlike existing theories, which are primarily based on straight-edge meshes, this study is the first to establish an analysis of the stability and optimal convergence of the finite volume method on curved-edge meshes. By constructing a dual mesh based on Gaussian points, we overcome the accuracy degradation issues caused by geometric deformation and Jacobian non-uniformity of curved-edge meshes. We prove the coercivity of the discrete bilinear form under weak mesh regularity conditions, thereby obtaining an optimal error estimate in the energy norm. Furthermore, using orthogonality and the Aubin-Nitsche technique, we derive an optimal $L^2$ error estimate. Numerical experiments cover problems with constant and anisotropic coefficients, different dual partition strategies, complex curved boundary domains, and interfaces with large deformations. Numerical results indicate that this method consistently achieves the optimal convergence order in both the $H^1$ and $L^2$ norms on a variety of curved-edge meshes. Compared to straight-edge meshes, curved-edge meshes offer significant advantages in approximating complex curved boundaries and demonstrate better resistance to distortion in cases involving sudden changes in coefficients and large deformations at interfaces. This paper provides a unified theoretical framework for the finite volume element method on curved-edge meshes and verifies the efficiency and robustness of the proposed method.

\end{abstract}


\end{frontmatter}

\section{Introduction}

Due to its inherent local conservation properties, the finite volume element method (FVEM) has been widely applied in various fields of science and engineering. To date, substantial and well-established research has been conducted on the stability analysis \cite{cwx12,cxz15,ll99,lf96,st93,zz12}, error estimates \cite{bwl16,clz02,ell02,hx98,lyz15,zz14}, and superconvergence \cite{czz13,czz15,cz94,ll12} of this method. However, these theories are limited to straight-edge meshes, and there remains a gap in the formulation and theoretical analysis of the method for curved-edged meshes. Therefore, the motivation of this paper is to develop a high-order finite volume element scheme suitable for curved-edged meshes and to analyze its stability and convergence.

Curved-edge meshes offer significant advantages when addressing physical problems involving complex geometries and multi-material interfaces, such as fluid-structure interaction \cite{ra11,t01}, multi-phase flow \cite{te04}, and problems with highly deformable interfaces \cite{kf18}. The computational domains for such problems typically contain curved boundaries and moving interfaces. To capture these geometric features, traditional straight-edge meshes often require excessive refinement or the adoption of piecewise linear approximations to approximate the true boundaries, both of which inevitably introduce additional geometric discretization errors. In contrast, curved-edge meshes can describe complex geometries more naturally and accurately, and better conform to boundaries and interfaces. This enhances the resolution of local features in the solution while reducing unnecessary degrees of freedom. Furthermore, when combined with high-order numerical methods, curved-edge meshes can more accurately capture gradient variations in the solution and the distribution at material interfaces, thereby improving the overall accuracy of the numerical simulation.

Although curved-edge meshes offer the advantages described above, to the best of our knowledge, no research has yet been conducted on developing a finite volume element method on them. In fact, constructing a high-precision finite volume element method on curved-edge meshes and establishing the corresponding theoretical analysis inevitably presents several key challenges. The first challenge is how to construct a finite volume element scheme that meets accuracy requirements when the mapping from the reference element to the general physical element is unknown. In practice, for a given curved-edge mesh, the mapping may or may not be known. When the mapping is known, we can directly construct a finite volume element scheme on the reference elements and then transform it to the curved-edge elements using the mapping. However, when the mapping is unknown, we need to design a reasonable approach to construct a high-precision finite volume element scheme. The second challenge is how to construct the corresponding dual mesh on a curved mesh. Research on the finite volume method on straight-edge meshes shows that the geometric quality of the dual elements directly affects numerical accuracy \cite{ll12,wll21,y06,zz14,zz15}. On curved-edge meshes, this degradation in accuracy is often more pronounced due to geometric mismatch errors caused by the non-uniformity of the Jacobian matrix. The third challenge is how to formulate appropriate regularity assumptions for curved-edge meshes to prove the stability of the finite volume element method. Most existing stability results for the FVEM on straight-edge meshes rely on mesh regularity assumptions, such as the $h^{1+\gamma}$ parallelogram condition \cite{lyz15,zz15}. However, such regularity conditions are difficult to apply to curved-edge meshes, and particularly when the mesh undergoes strong geometric deformation, the coercivity of the discrete bilinear form cannot be proven.

To overcome the aforementioned challenges, we propose the following strategies in current paper. When the mapping from reference elements to curved-edge elements is unknown, we construct an appropriate interpolation mapping based on the geometric curvature characteristics of the curved-edge elements to generate approximate curved-edge elements that approximate the true geometry, and establish a finite volume element scheme on these approximate curved-edge elements. Regarding the construction of dual meshes, the complex boundary topologies and irregular geometries of curved-edge elements typically make it impossible to construct dual elements directly on them. To address this, we first construct a dual partition based on Gaussian points on the reference element, and then map it to the curved-edge elements. This approach circumvents the difficulties posed by the complex shapes of curved-edge elements while ensuring computational accuracy. To address the coercivity of a discrete bilinear form, since geometric constraints for curved-edge meshes cannot be directly specified, we propose a regularity condition derived by constraining the derivatives of the Jacobian matrix. Furthermore, unlike previous work \cite{lyz15}, we employ a strategy that combines the orthogonality property of the discrete scheme with the Aubin–Nitsche technique, providing a concise proof of the optimal $L^2$ error estimate. Finally, we designed three sets of numerical experiments to validate the theoretical results.

The paper is organized as follows. We provide the construction of the curved-edge mapping and present the regularity assumptions in Section 2. The finite volume element scheme based on Gaussian dual partition is constructed in Section 3. Section 4 is devoted to the coercivity of the bilinear form. We derive the optimal $H^1$- and $L^2$- error estimates in Section 5. Some numerical examples are shown in Section 6 to validate the theoretical results. Finally, the appendix summarizes the derivative relationships in the mapping from the reference element to curved-edge elements.

Throughout the paper, we  adopt standard Sobolev space notations. For a domain $D$, let $W^{m,p}(D)$ be the Sobolev space equipped with the norm $\parallel\cdot \parallel_{m,p,D}$ and semi-norm $|\cdot|_{m,p,D}$. When $p=2$,  we write
 $W^{m,p}(D)=H^m(D)$, $\parallel\cdot \parallel_{m,2,D}=\parallel\cdot \parallel_{m,D}$ and $|\cdot|_{m,2,D}=|\cdot|_{m,D}$. In particular, $\parallel\cdot\parallel_{0,D}$ stands for the $L^2$ norm in $D$. The notation $\|\cdot\|_{\infty,D}$ is used to denote the infinity norm for functions or matrices. In what follows, $C$ represents any positive constant that independent of mesh size $h$, varying per occurrence.

\section{Preliminary}
In this section, we focus primarily on the construction of interpolation mappings for curved-edge elements when the mapping is unknown, and present regularity assumptions on the Jacobian matrix of curved-edge meshes. 

For a given curved-boundary domain $\Omega\subset\mathbb{R}^2$, let $\mathcal{T}$ be a family of  curved-edge quadrilateral partitions of $\Omega$, consisting of element $K$'s with diameter $h_{K}$. Denote $h=\max_{K\in\mathcal{T}}\{h_K\}$. For each exact element $K\in\mathcal{T}$, let $K^h$ be its geometric interpolation approximation. The family of all such approximate elements constitutes the approximate partition $\mathcal{T}_h$, and the corresponding approximate computational domain is $\Omega_h=\bigcup\limits_{K^h \in \mathcal{T}_h} K^h$, which serves as an approximation to the domain $\Omega$. We assume that the element mappings \(\Psi_K\) are compatible on common edges, namely, the edge parametrizations induced by two neighboring elements coincide up to orientation. In general, we have $\Omega\not\subset\Omega_h$ and $\Omega_h\not\subset\Omega$, but we can ensure that $\partial\Omega_h$ is sufficiently close to $\partial\Omega$ as the mesh is refined.  

For a given curved-edge mesh, the mapping $\Psi_K$ from the reference element $\widehat{K}=[-1,1]\times[-1,1]$ to physical element $K^h$ can be determined in three cases.
\begin{enumerate}
\item[Case I.] The mapping $\Psi_K$ from the reference element to every each curved-edge element is known.   
\item [Case II.] Each edge of the curved-edge element are explicitly known through parametric representations. According to \cite{pg19}, one can build an exact mapping $\Psi_K$ to fit each boundary of the curved-edge element. 
\item [Case III.] The vertex coordinates of the curved-edge element $K$ are known, and at least one of its edges has only an implicit equation and no explicit expression.
    In this case, one can construct an interpolation mapping $\Psi_K$ of polynomial form using vertices, certain points extracted from the implicit edges, and certain points inside the element \cite{cr72}.

\end{enumerate}

\begin{remark}
For case II, the parametric equation of each edge of the curved-edge element $K$ is known, the most standard way to construct the mapping $\Psi_K$ from the reference element to the curved-edge element $K^h$ is to use the edge parametric equation directly. To be specific, assume that the four vertices of the curved-edge quadrilateral are $P_{1}, P_{2}, P_{3}, P_{4}$, and its four edges are $\bm{r}_1(\xi), \bm{r}_2(\xi), \bm{r}_3(\xi), \bm{r}_4(\xi)$ for $\xi \in [0,1]$ with $\bm{r}_i(\cdot)=(x(\cdot),y(\cdot))$, $i=1,2,3,4$.
Each edge parametrization satisfies endpoint compatibility conditions $\bm{r}_i(0) = P_{i}$ and $\bm{r}_i(1) = P_{i+1}$ with $P_5=P_1$. Even if the four edge equations of the curved element are known, the mapping from the reference element to the physical element is not unique. Since the boundary conditions constrain only the edges of the physical element, there are infinitely many possible locations in physical space for the image of each point inside the reference element, provided no additional internal extension rules are applied. To eliminate this uncertainty, we can use the Coons extrapolation method to uniquely construct the internal mapping, which can be expressed as
\begin{equation*}
\Psi_K(\xi,\eta)=(1-\eta)\bm{r}_1(\xi)+\eta \bm{r}_3(1-\xi)+(1-\xi)\bm{r}_4(1-\eta)+\xi \bm{r}_2(\eta)-C(\xi,\eta),
\end{equation*}
where the correction term $C(\xi,\eta)$ is
\begin{equation*}
C(\xi,\eta)=(1-\xi)(1-\eta)P_{1}+\xi(1-\eta)P_{2}+\xi\eta P_{3}+(1-\xi)\eta P_{4}.
\end{equation*}
From the endpoint compatibility conditions, one can easily get
\begin{equation*}
\Psi_K(\xi,0)=\bm{r}_1(\xi),\Psi_K(1,\eta)=\bm{r}_2(\eta),\Psi_K(\xi,1)=\bm{r}_3(1-\xi), \Psi_K(0,\eta)=\bm{r}_4(1-\eta),
\end{equation*}
which implies that the constructed mapping fits the curved edges of the physical element $K^h$ exactly, i.e., $K=K^h$.
\end{remark}

\begin{remark}
For Case III, we present a procedure for constructing the target mapping; see also \cite{cr72}. 
Let $K$ be a curved-edge quadrilateral element. We assume that only the implicit equations
$\zeta_\ell(x,y)=0$ $(\ell=1,2,3,4)$
of its four curved edges and the four vertices of $K$ are known. Our goal is to construct an approximate curved-edge element $K^h$ that approximates $K$, and defined a bi-$k$ polynomial mapping from the reference element $\widehat K$ onto $K^h$. The construction is divided into the following three steps.
\begin{enumerate}
\item[\textbf{1)}] \textbf{Construction of the associated straight-edge element.}
We first connect the four vertices of the curved-edge element $K$ in sequence and obtain a straight-edge quadrilateral element $K'$. Let $
\Lambda_K:\widehat K\to K'$
be the bilinear mapping from the reference element $\widehat K$ onto $K'$. This mapping serves as the underlying straight-edge mapping; see Figure~\ref{fig30}.

\item[\textbf{2)}] \textbf{Determination of interpolation nodes.}
We choose the standard tensor-product Lagrange interpolation nodes $\{\widehat z_{ij}\}_{0\le i,j\le k}\subset \widehat K$ to construct the bi-$k$ polynomial mapping. Their images under the bilinear mapping $\Lambda_K$ are denoted by
$z_{ij}^{'}=\Lambda_K(\widehat z_{ij})$ $(0\le i,j\le k)$, which belong to the corresponding straight edge element $K'$. Then, for each boundary interpolation point $z^{'}_{ij}$, we establish a line equation $l_{ij}(x,y)=0$ that passes through $z_{ij}^{'}$ and perpendicular to the corresponding straight edge. By solving the system of equations formed by the line equation and the implicit equation of the curved edge, we obtain the boundary interpolation point $z_{ij}$.
For example, a bi-cubic mapping as shown in Figure~\ref{fig30}, to determine the interpolation point $z_{31}$ on the curve $\zeta_3$, we solve the system of equations
\[
\begin{cases}
l_{31}(x,y)=0,\\
\zeta_3(x,y)=0.
\end{cases}
\]

For the interior interpolation nodes of $K^h$, we set $
z_{ij}=z_{ij}^{'}$.
In this way, all interpolation points $\{z_{ij}\}_{0\le i,j\le k}$ for the curved-edge approximation are determined.

\item[\textbf{3)}] \textbf{Construction of the target mapping.}
Let $\{\widehat L_{ij}\}_{0\le i,j\le k}$ be the tensor-product Lagrange basis functions on $\widehat K$ associated with the nodes $\{\widehat z_{ij}\}_{0\le i,j\le k}$. We define the target mapping
$\Psi_K:\widehat K\to K^h$
by
\[
\Psi_K(\xi)
=
\sum_{i=0}^{k}\sum_{j=0}^{k}
z_{ij}\widehat L_{ij}(\xi).
\]
Since $z^{'}_{ij}=\Lambda_K(\widehat{z}_{ij})$, we may rewrite the mapping as $\Psi_{K}(\xi)=\Lambda_K(\xi)+\Phi_K(\xi)$, where 
\[
\Phi_K(\xi)
=
\sum_{i=0}^{k}\sum_{j=0}^{k}
(z_{ij}-z^{'}_{ij})\widehat L_{ij}(\xi)
\]
 is the non-affine perturbation caused by replacing the straight edges of $K'$ with the true curved edges of $K^h$.
\end{enumerate}
\begin{figure}[htbp]  
    \centering  
    \includegraphics[width=0.8\linewidth, keepaspectratio]{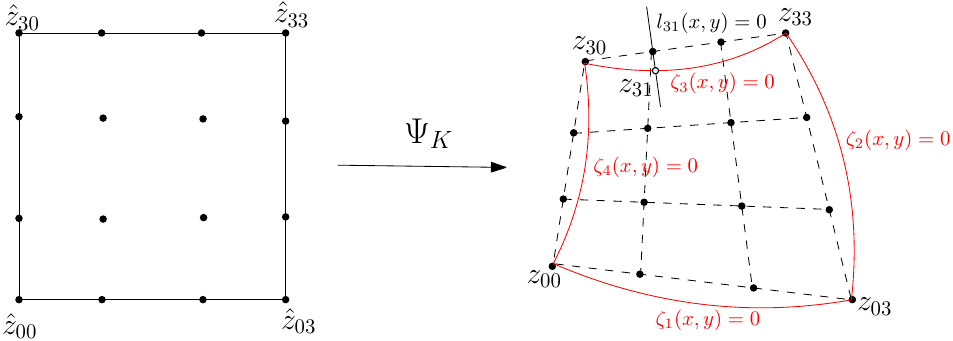}
    \caption{Distribution of interpolation points for quadrilateral element when $k=3$}  
    \label{fig30}  
\end{figure}
\end{remark}

To establish the asymptotic error estimates on curved elements, Ciarlet and Raviart \cite{cr72} proved several essential conditions on the element mapping, which we summarize below.

\begin{lemma}\label{lemma11}
Let \(\Psi_K:\widehat K\to K^h\) be a mapping from the reference element \(\widehat K\) to the physical curved quadrilateral element \(K^h\). Assume that the following conditions hold:
\begin{enumerate}[label=(\roman*), leftmargin=*]
\item The elements \(K'\) form a regular family of quadrilaterals, where \(K'\) is the straight-edge quadrilateral sharing the same four vertices as \(K^h\);
\item The distance between each non-vertex node of \(K^h\) and the corresponding node of \(K'\) tends to zero as \(h_K\to0\);
\item There exists a constant \(C>0\), independent of \(h_K\), such that
\begin{equation}\label{eq:mapp_deriv_bound}
\|D^s\Psi_K\|_{\infty,\widehat K}\le C h_K^s \quad \text{for all } 1\le s\le k+1.
\end{equation}
\end{enumerate}
Then \(\Psi_K\) is a \(C^{k+1}\)-diffeomorphism, and the interpolation estimate
\begin{equation}\label{eq:interp_error}
|v-\Pi_h^k v|_{s,K^h}\le C h_K^{k+1-s}|v|_{k+1,K^h}
\end{equation}
holds for all \(v\in H^{k+1}(K^h)\) and all integers \(s\) with \(0\le s\le k+1\), where \(\Pi_h^k\) is the \(k\)-th order Lagrange interpolation operator.
\end{lemma}

By Lemma~\ref{lemma11}, the interpolation estimate
\eqref{eq:interp_error} can be applied provided that the element mapping
$\Psi_K$ satisfies the regularity condition
\eqref{eq:mapp_deriv_bound}. In addition, we need to provide estimates for several geometric quantities associated with
$\Psi_K$. Let $\mathbb J_K=D\Psi_K$ be the Jacobian matrix of the mapping $\Psi_K$,
whose determinant is $J_K$. Let $r_\xi=\sqrt{x_\xi^2+y_\xi^2}$, $r_\eta=\sqrt{x_\eta^2+y_\eta^2}$, and $s=x_\xi x_\eta+y_\xi y_\eta$.

\begin{lemma}
\label{lem:geometric_coefficient_bounds}
Assume that the element mapping $\Psi_K$ satisfies
\begin{equation}
\label{eq:mapping_scaled_bounds}
\|D^m\Psi_K\|_{\infty,\widehat K}
\le Ch^m,
\qquad m=1,2,3,
\end{equation}
and that the Jacobian determinant satisfies the uniform
non-degeneracy condition
\begin{equation}
\label{eq:jacobian_lower_bound}
ch^2
\le
|J_K(\widehat{\boldsymbol x})|
\le
Ch^2,
\qquad
\widehat{\boldsymbol x}\in\widehat K,
\end{equation}
where $c,C>0$ are independent of $h$.
Then
\begin{align}
&\|J_K-\overline J_K\|_{\infty,\widehat K}
\le Ch^3,
\label{eq:JK_bar_est}
\\
&\bigl\|D(r_\eta^2J_K^{-1})\bigr\|_{\infty,\widehat K}
+
\bigl\|D(sJ_K^{-1})\bigr\|_{\infty,\widehat K}
+
\bigl\|D(r_\xi^2J_K^{-1})\bigr\|_{\infty,\widehat K}
\le Ch,
\label{eq:metric_coeff_est}
\\
&\|\mathbb J_K^{-1}
-\overline{\mathbb J_K^{-1}}\|_{\infty,\widehat K}
\le C,
\label{eq:Jinv_bar_est}
\\
&\|D\mathbb J_K^{-1}\|_{\infty,\widehat K}
\le C,
\label{eq:DJinv_est}
\\
&\bigl\|
D\bigl(
\mathbb J_K^{-1}
-\overline{\mathbb J_K^{-1}}
\bigr)
\bigr\|_{\infty,\widehat K}
\le C,
\label{eq:DJinv_bar_est}
\end{align}
where $\overline {J_K}$ and $\overline{\mathbb J^{-1}_K}$ denote the mean values of $J_K$ and $\mathbb J^{-1}_K$ over $\widehat K$.
\end{lemma}

\begin{proof}
By \eqref{eq:mapping_scaled_bounds}, we have
\begin{equation}
\label{eq:basic_geo_bounds}
\|\mathbb J_K\|_{\infty,\widehat K}
\le Ch,
\qquad
\|D\mathbb J_K\|_{\infty,\widehat K}
\le Ch^2,
\qquad
\|D^2\mathbb J_K\|_{\infty,\widehat K}
\le Ch^3.
\end{equation}
Since $J_K=\det\mathbb J_K$, Leibniz's formula gives
\begin{equation}
\label{eq:det_bounds}
\|J_K\|_{\infty,\widehat K}
\le Ch^2,\quad\|J_K^{-1}\|_{\infty,\widehat K}
\le Ch^{-2}, 
\quad
\|DJ_K\|_{\infty,\widehat K}
\le Ch^3,
\end{equation}
which together with $D(J_K^{-1})=-J_K^{-2}(DJ_K)$ yields
\begin{equation}
\label{eq:det_inverse_bounds}
\|D(J_K^{-1})\|_{\infty,\widehat K}
\le Ch^{-1}.
\end{equation}
By $W^{1,\infty}$-Poincar\'e inequality, we have
\[
\|J_K-\overline J_K\|_{\infty,\widehat K}
\le
C\|DJ_K\|_{\infty,\widehat K}
\le Ch^3,
\]
which proves \eqref{eq:JK_bar_est}. Obviously, one has
\[
\|r_\xi^2\|_{\infty,\widehat K}
+
\|r_\eta^2\|_{\infty,\widehat K}
+
\|s\|_{\infty,\widehat K}
\le Ch^2.
\]
A direct calculation yields
\[
\|D(r_\xi^2)\|_{\infty,\widehat K}
+
\|D(r_\eta^2)\|_{\infty,\widehat K}
+
\|Ds\|_{\infty,\widehat K}
\le Ch^3,
\]
which together with Leibniz's formula and \eqref{eq:det_inverse_bounds} gives
\begin{align*}
&
\|D(r_\eta^2J_K^{-1})\|_{\infty,\widehat K}
+
\|D(sJ_K^{-1})\|_{\infty,\widehat K}
+
\|D(r_\xi^2J_K^{-1})\|_{\infty,\widehat K}
\\
&\qquad
\le
C\left(
h^3h^{-2}+h^2h^{-1}
\right)
\le Ch,
\end{align*}
which proves \eqref{eq:metric_coeff_est}. Since
\[
D(\mathbb J_K^{-1})
=
-\mathbb J_K^{-1}
(D\mathbb J_K)
\mathbb J_K^{-1},
\]
thus we obtain
\[
\|D\mathbb J_K^{-1}\|_{\infty,\widehat K}
\le
C h^{-1}h^2h^{-1}
\le C,
\]
which proves \eqref{eq:DJinv_est}. Applying the $W^{1,\infty}$-Poincar\'e inequality
 to $\mathbb J_K^{-1}$ gives
\[
\|\mathbb J_K^{-1}
-\overline{\mathbb J_K^{-1}}\|_{\infty,\widehat K}
\le
C\|D\mathbb J_K^{-1}\|_{\infty,\widehat K}
\le C,
\]
which proves \eqref{eq:Jinv_bar_est}.
Since $\overline{\mathbb J_K^{-1}}$ is a constant matrix, so we have
\[
\bigl\|
D\bigl(
\mathbb J_K^{-1}
-\overline{\mathbb J_K^{-1}}
\bigr)
\bigr\|_{\infty,\widehat K}=\bigl\|
D\mathbb J_K^{-1}
\bigr\|_{\infty,\widehat K}
\le C.
\]
\end{proof}

\section{Finite volume element scheme}

We consider the following elliptic boundary value problem
\begin{subequations}\label{model}
\begin{align}
-\nabla\cdot\bigl(\kappa(x,y)\nabla u\bigr)
&=f,
\qquad {\rm in}\ \Omega,
\label{model_1}\\
\kappa(x,y)\frac{\partial u}{\partial\bm n}
+\sigma(x,y)u
&=g,
\qquad {\rm on}\ \Gamma=\partial\Omega.
\label{model_2}
\end{align}
\end{subequations}
Here $\bm n$ denotes the outward unit normal vector to $\Gamma$,
$f\in L^2(\Omega)$ is the source term, and the diffusion coefficient
$\kappa$ satisfies
\[
\kappa\in W^{1,\infty}(\Omega),
\qquad
\kappa(x,y)\ge\kappa_0>0
\quad\text{in }\Omega.
\]
The Robin coefficient satisfies
\[\sigma\in W^{1,\infty}(\Omega),\qquad
0<\sigma_0
\le
\sigma(x,y)
\le
\sigma_1
\qquad\text{on }\Gamma.
\]

Define trial function space as
\begin{equation*}
\mathcal{U}^k_h=\{u_h\in C(\bar{\Omega}_h)|\ \widehat{u}_h=u_h|_{K^h} \circ \Psi_{K} \in \mathbb{Q}_k(\widehat{K}),\ \forall K^h\in \mathcal T_h \},
\end{equation*}
where $\mathbb{Q}_k(\widehat{K})$ denotes the space of all bi-$k$ polynomials on $\widehat{K}$.

Next, we introduce the dual mesh and the corresponding test function space.
Let $\mathbb Z_k=\{1,2,\ldots,k\}$ and $\mathbb Z_k^0=\{0,1,\ldots,k\}$. We specify the Gaussian points and the interpolation points on the reference element $\widehat K$, and then transfer them to curved-edge element by the mapping $\Psi_K$.

Denote by \(g_i\), \(i\in\mathbb Z_k\), the \(k\)-point Gauss--Legendre points on \([-1,1]\). That is,
\[
-1<g_1<g_2<\cdots<g_k<1,
\]
and \(\{g_i\}_{i\in\mathbb Z_k}\) are the roots of the Legendre polynomial
\(\mathbb P_k\). The associated Gauss-Legendre quadrature rule is exact for all
polynomials of degree at most \(2k-1\). For convenience, we extend the Gauss--Legendre points by adding two
endpoints $g_0=-1$ and $g_{k+1}=1$. Then the extended tensor-product point set on \(\widehat K=[-1,1]^2\) is
defined by
\[
g^{\widehat K}_{i,j}=(g_i,g_j), \quad \forall\, i,j\in\mathbb Z^0_{k+1}.
\]
Here the points with \(i,j\in\mathbb Z_k\) are the standard Gauss--Legendre
points, while the remaining points are auxiliary boundary points used to close
the dual subregions.

For each element $K^h\in \mathcal T_h$, the images of the above Gaussian points under the mapping $\Psi_K$ are defined as
\[
G^{K^h}_{i,j}=\Psi_K\big(g^{\widehat K}_{i,j}\big), \quad \forall\, i,j\in\mathbb Z^0_{k+1}.
\]
The set of all Gaussian points on the element $K^h$ is denoted by
\[
\mathcal G_{K^h}=\big\{G^{K^h}_{i,j} \,\big|\, i,j\in\mathbb Z^0_{k+1}\big\},
\]
and the global set of Gaussian points over the computational domain $\Omega_h$ is given by
\[
\mathcal G=\bigcup_{K^h\in\mathcal T_h}\mathcal G_{K^h}.
\]
Similarly, let $l^{\widehat K}_{i,j}$ ($i,j\in\mathbb Z_k^0$) denote the Lagrangian interpolation points on $\widehat K$, which will be used to define the trial functions. Their images on the physical element $K^h$ are denoted by
\[
L^{K^h}_{i,j}=\Psi_K\big(l^{\widehat K}_{i,j}\big), \quad \forall\, i,j\in\mathbb Z_k^0.
\]
Accordingly, the set of Lagrangian interpolation points on the element $K^h$ is
\[
\mathcal N_{K^h}=\big\{L^{K^h}_{i,j} \,\big|\, i,j\in\mathbb Z_k^0\big\},
\]
and the global set of Lagrangian interpolation points over $\mathcal T_h$ is defined by
\[
\mathcal N=\bigcup_{K^h\in\mathcal T_h}\mathcal N_{K^h}.
\]
Let $\mathcal N^0=\mathcal N\setminus\partial\Omega_h$ be the set of all internal interpolation points and  $N^b=\mathcal N\backslash\mathcal N^0$ denotes the set of all boundary interpolation points. Figure ~\ref{fig:1}(a) shows the Gaussian points and Lagrangian interpolation points on $\widehat K$ when $k=2$. 

We next construct the dual partition associated with the above interpolation points. To this end, we once again begin by defining the dual element from the reference element $\widehat K$. We use the straight lines $\xi = g_i$ and $\eta = g_j$ (where $i, j \in \mathbb{Z}_{k+1}^0$) to partition the reference element $\widehat{K}$ into $(k+1)^2$ subregions with no common interior. As shown in Figure \ref{fig:1}(b), for each interpolation point $l^{\widehat K}_{i,j}$ ($i, j \in \mathbb{Z}_k^0$), there is a unique subregion
\[
\widehat K^*_{i,j}=[g_i,g_{i+1}]\times[g_j,g_{j+1}], \quad \forall\, i,j\in\mathbb Z_k^0,
\]
which is associated with $l^{\widehat K}_{i,j}$. Then we map these reference subregions onto the curved-edge elements $K^h$ as follows
\[
K_{i,j}^{*,h}=\Psi_K\big(\widehat K^*_{i,j}\big), \quad \forall\, i,j\in\mathbb Z_k^0.
\]
Since $L^{K^h}_{i,j}=\Psi_K\big(l^{\widehat K}_{i,j}\big)$, the local dual subregions $K_{i,j}^{*,h}$ corresponds naturally to the physical interpolation point $L^{K^h}_{i,j}$.  Then one can combine the local dual subregions with respect to $P \in \mathcal{N}$ into a complete control element, that is,
\[
 K_P^{*,h}=\bigcup\big\{K_{i,j}^{*,h} \,\big|\, K^h\in\mathcal T_h,\ L^{K^h}_{i,j}=P\big\}.
\]
The dual partition of $\mathcal T_h$ is therefore given by
\[
\mathcal T_h^*=\big\{ K_P^{*,h} \,\big|\, P\in\mathcal N\big\}.
\]
Based on the dual partition $\mathcal T_h^*$, we define the piecewise-constant test function space by
\[
\mathcal V_{h}=\left\{ v_h |\ v_h=\sum_{P\in\mathcal N} v_{K_P^{*,h}}\psi_{K_P^{*,h}} \right\},
\]
where $v_{K_P^{*,h}}$ and $\psi_{K_P^{*,h}}$ respectively denote the constant and characteristic functions on the control set $K_P^{*,h}$. 

\begin{figure}[htbp]
    \centering
    \begin{subfigure}[b]{0.50\textwidth}
        \centering
        \includegraphics[height=4.8cm, width=4.8cm]{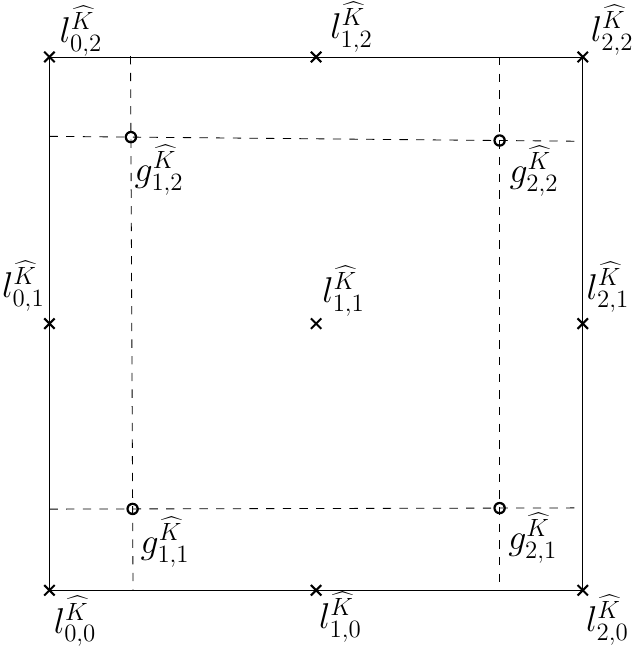}
        \caption{Gaussian and Interpolation points }
    \end{subfigure}
    \hfill
    \begin{subfigure}[b]{0.45\textwidth}
        \centering
        \includegraphics[height=4.8cm, width=4.8cm]{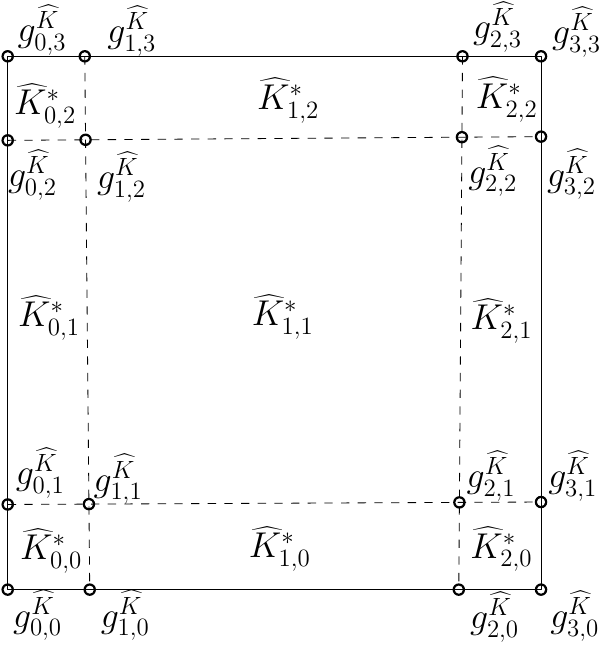}
        \caption{Distribution of dual subregions}
    \end{subfigure}
    \vspace{0.5em} 
    \caption{Gaussian/Interpolation points and distribution of dual subregions on $\widehat{K}$}
    \label{fig:1}
\end{figure}

Under the mapping \(\Psi_K\) constructed as described above, the computational
domain \(\Omega_h\) and the physical domain
\(\Omega\) may not coincide, which can result in the cases where \(\Omega\not\subset \Omega_h\)
or \(\Omega_h\not\subset \Omega\). Therefore, the source term \(f\), originally defined on \(\Omega\),
must be evaluated on the computational domain \(\Omega_h\). 
To avoid this mismatch, we choose a bounded domain \(\widetilde\Omega\) that is independent of $h$, such that $\Omega\cup\Omega_h\subset\widetilde{\Omega}$.
 Then we define $\widetilde f$ as the zero extension of $f$ from
$\Omega$ to $\widetilde\Omega$. Moreover, let $\widetilde \kappa$, $\widetilde\sigma$ and $\widetilde g$ be bounded extensions
of $\kappa$, $\sigma$ and $g$, respectively. We simply refer to them as $f_h=\widetilde f|_{\Omega_h}$, 
$\sigma_h=\widetilde\sigma|_{\Gamma_h}$, $g_h=\widetilde g|_{\Gamma_h}$ and $\kappa_h=\widetilde \kappa|_{\Gamma_h}$, where $\Gamma_h=\partial\Omega_h$.

Based on the above setting, the finite volume element method is to
find $u_h\in\mathcal U_h^k$ such that
\begin{equation}\label{scheme}
a_h(u_h,v_h)
=
(f_h,v_h)
+
\langle g_h,v_h\rangle,
\qquad
\forall v_h\in\mathcal V_h,
\end{equation}
where
\begin{equation}\label{robin-ah-decomp}
a_h(u_h,v_h)
=
a_{h,0}(u_h,v_h)
+
r_h(u_h,v_h),
\end{equation}
with
\begin{align*}
a_{h,0}(u_h,v_h)
&=
-
\sum_{P\in\mathcal N}
v_{K_P^{*,h}}
\int_{\partial_0K_P^{*,h}}
\kappa_h(x,y)\frac{\partial u_h}{\partial\bm n}
\,{\rm d}s,\\
r_h(u_h,v_h)
&=
\sum_{P\in\mathcal N^b}
v_{K_P^{*,h}}
\int_{\Gamma_{P,h}}
\sigma_hu_h\,{\rm d}s,\\
\langle g_h,v_h\rangle
&=
\sum_{P\in\mathcal N^b}
v_{K_P^{*,h}}
\int_{\Gamma_{P,h}}
g_h\,{\rm d}s\\
(f,v_h)&=\sum_{P\in\mathcal N}v_{K_P^{*,h}}\int_{K_P^{*,h}}f_h {\rm d}x{\rm d}y.
\end{align*}
Here $\Gamma_{P,h}=\partial K_P^{*,h}\cap\Gamma_h$ and $\partial_0K_P^{*,h}=\partial K_P^{*,h}\setminus\Gamma_{P,h}$.

\section{Stability Analysis}

In this section, we prove the coercivity of the Robin finite volume
bilinear form $a_h(\cdot,\cdot)$.

Let $\mathcal E_{\mathcal T_h^*}$ be the set of all internal dual
edges. For any $v_h\in\mathcal V_h$, define
\[
|v_h|_{\mathcal T_h^*}^2
=
\sum_{e\in\mathcal E_{\mathcal T_h^*}}
h_e^{-1}
\int_e[v_h]_e^2\,{\rm d}s,
\]
where $h_e$ denotes the length of $e$ and
$
[v_h]_e
=
v_h|_{K_2^{*,h}}
-
v_h|_{K_1^{*,h}}
$
denotes the jump across the dual edge
$e=K_1^{*,h}\cap K_2^{*,h}$. A direct calculation yields
\begin{equation}\label{bil}
a_{h,0}(u_h,v_h)
=
\sum_{e\in\mathcal E_{\mathcal T_h^*}}
[v_h]_e
\int_e
\kappa(x,y)
\frac{\partial u_h}{\partial\bm n_e}
\,{\rm d}s,
\end{equation}
where $\bm n_e$ denotes the prescribed unit normal vector on $e$. For any $u_h\in\mathcal U^k_h$ and $v_h\in\mathcal V_h$, we introduce the following energy norms
\begin{equation}\label{robin-energy-norm}
|||u_h|||^2_{\Omega_h}
:=
|u_h|_{1,\Omega_h}^2
+
\|\sigma_h^{1/2}u_h\|_{0,\Gamma_h}^2,
\end{equation}
and 
\begin{equation}\label{robin-test-norm}
\|v_h\|_{\mathcal T_h^*}^2
:=
|v_h|_{\mathcal T_h^*}^2
+
\|\sigma_h^{1/2}v_h\|_{0,\Gamma_h}^2,
\end{equation}
where
\[
\|v_h\|_{0,\Gamma_h}^2
=
\sum_{P\in\mathcal N^b}
\int_{\Gamma_{P,h}}
|v_h|^2\,{\rm d}s.
\]
By using trace inequality and Poincar\'e inequality, we obtain that the norm $|||\cdot|||_{\Omega_h}$ is equivalent to the standard $H^1$ norm.

\subsection{Coercivity}

Let
$\overline{g^{\widehat K}_{i,j}g^{\widehat K}_{i+1,j}}$
be the line segment connecting
$g^{\widehat K}_{i,j}$ and
$g^{\widehat K}_{i+1,j}$. The corresponding dual edges on $K^h$ are
\[
E_{i,j}^{K^h,x}
=
\Psi_K
\left(
\overline{
g^{\widehat K}_{i,j}
g^{\widehat K}_{i+1,j}
}
\right),
\qquad
(i,j)\in
\mathbb Z_k^0\times\mathbb Z_k,
\]
and
\[
E_{i,j}^{K^h,y}
=
\Psi_K
\left(
\overline{
g^{\widehat K}_{i,j}
g^{\widehat K}_{i,j+1}
}
\right),
\qquad
(i,j)\in
\mathbb Z_k\times\mathbb Z_k^0.
\]
Thus, the set of all dual edges within $K^h$ is
\[
\mathcal E_{\mathcal T_h^*}\cap K^h=\{E_{i,j}^{K^h,x}|(i,j)\in\mathbb Z^0_k\times\mathbb Z_k\}\cup\{E_{i,j}^{K^h,y}|(i,j)\in\mathbb Z_k\times\mathbb Z_k^0\}.
\]
For $v_h\in\mathcal V_h$, let
$v_{i,j}^{K^h}=v_h(L_{i,j}^{K^h})$, and define the jump of $v_h$ across dual edges $E_{i,j}^{K^h,x}$ and $E_{i,j}^{K^h,y}$ as
\[
[v_h]_{i,j}^{K^h,x}
=
v_{i,j}^{K^h}
-
v_{i,j-1}^{K^h},
\qquad
(i,j)\in
\mathbb Z_k^0\times\mathbb Z_k,
\]
\[
[v_h]_{i,j}^{K^h,y}
=
v_{i,j}^{K^h}
-
v_{i-1,j}^{K^h},
\qquad
(i,j)\in
\mathbb Z_k\times\mathbb Z_k^0,
\]
and the jump of $v_h$ at Gaussian point $G^{K^h}_{i,j}$ as
\[
\lfloor v_h\rfloor_{i,j}^{K^h}
=
v_{i,j}^{K^h}
+
v_{i-1,j-1}^{K^h}
-
v_{i,j-1}^{K^h}
-
v_{i-1,j}^{K^h},
\qquad
(i,j)\in
\mathbb Z_k\times\mathbb Z_k.
\]
Obviously, we have
\begin{equation}\label{id1}
\lfloor v_h\rfloor_{i,j}^{K^h}
=
[v_h]_{i,j}^{K^h,x}
-
[v_h]_{i-1,j}^{K^h,x}
=
[v_h]_{i,j}^{K^h,y}
-
[v_h]_{i,j-1}^{K^h,y}.
\end{equation}
For each curved dual edge, we prescribe a unit normal vector field \(\bm{\nu}\). If the dual edge is logically horizontal, the direction of
\(\bm{\nu}\) is taken to be logically upward. If the dual edge is logically
vertical, the direction of \(\bm{\nu}\) is taken to be logically to the
right, see Figure \ref{fig:aaa}. Using these notations, the bilinear form \eqref{bil} can be decomposed into
\begin{equation}\label{ee1}
a_{h,0}(u_h,v_h)
=
\sum_{K^h\in\mathcal T_h}
a_{h,0,K^h}(u_h,v_h),
\end{equation}
where
\begin{align}
a_{h,0,K^h}(u_h,v_h)
={}&
\sum_{(i,j)\in\mathbb Z_k^0\times\mathbb Z_k}
[v_h]_{i,j}^{K^h,x}
\int_{E_{i,j}^{K^h,x}}
\kappa_h(x,y)
\frac{\partial u_h}{\partial\bm\nu}
\,{\rm d}s
\nonumber\\
&+
\sum_{(i,j)\in\mathbb Z_k\times\mathbb Z_k^0}
[v_h]_{i,j}^{K^h,y}
\int_{E_{i,j}^{K^h,y}}
\kappa_h(x,y)
\frac{\partial u_h}{\partial\bm\nu}
\,{\rm d}s.
\label{local-diffusion-form}
\end{align}
\begin{figure}[htbp]
    \centering
        \centering
        \includegraphics[height=4cm]{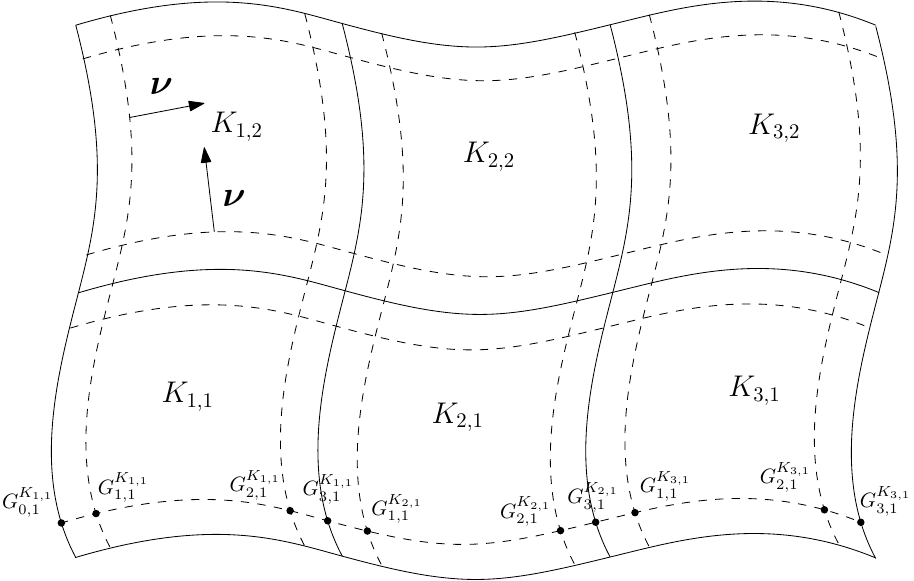}
    \caption{Specifying the direction of the unit normal vector corresponding to the dual edge and a global dual edge along the $x$-direction when $k = 2$}
    \label{fig:aaa}
\end{figure}
Let \(\Pi^k_h:H^2(\widetilde{\Omega})\to \mathcal U^k_h\) be the
\(k\)-th order Lagrange interpolation operator. We also introduce a dual interpolation operator 
$
\Pi_{h}^{k,*}:
\mathcal U_h^k
\longrightarrow
\mathcal V_h.
$ For each $u_h\in\mathcal U_h^k$, let $\widehat u_h=u_h|_{K^h}\circ\Psi_K$. For any $v_h=\Pi_{h}^{k,*}u_h$, we have $v_h|_{K^h}=\sum_{(i,j)\in\mathbb Z_k\times\mathbb Z_k}v^{K^h}_{i,j}\psi^{K^h}_{i,j}$, where the coefficients $v_{i,j}^{K^h}$ satisfy
\begin{equation}\label{pro}
\left\{
\begin{aligned}
&v_{0,0}^{K^h}=\widehat u_h(-1,-1),\\
&v_{i,0}^{K^h}-v_{i-1,0}^{K^h}=A_i\frac{\partial\widehat u_h}{\partial\xi}(g_i,-1),\qquad &&i\in\mathbb Z_k,\\
&v_{0,j}^{K^h}-v_{0,j-1}^{K^h}=A_j\frac{\partial\widehat u_h}{\partial\eta}(-1,g_j),\qquad &&j\in\mathbb Z_k,\\
&\lfloor v_h\rfloor_{i,j}^{K^h}
=
A_iA_j
\frac{\partial^2\widehat u_h}
{\partial\xi\partial\eta}
(g_i,g_j),
\qquad
&&(i,j)\in
\mathbb Z_k\times\mathbb Z_k,
\end{aligned}
\right.
\end{equation}
where \(A_i,i\in\mathbb Z_k\) denote the weights of the Gauss-Legendre quadrature $\sum_{i=1}^{k}A_ip(g_i)$ for computing the integral $\int_{-1}^{1}p(x){\rm d}x$.
\begin{lemma}\label{lemma:robin-pi-existence}
For every $u_h\in\mathcal U_h^k$, there exists a unique
$v_h\in\mathcal V_h$ satisfying \eqref{pro}. Hence the 
interpolation operator
$\Pi_{h}^{k,*}:\mathcal U_h^k\rightarrow\mathcal V_h$
is well defined.
\end{lemma}

\begin{proof}
It follows from \eqref{pro} that we have
\begin{align}
v_{i,0}^{K^h}-v_{0,0}^{K^h}
=
\sum_{r=1}^{i}
\left(
v_{r,0}^{K^h}-v_{r-1,0}^{K^h}
\right)=
\sum_{r=1}^{i}
A_r
\frac{\partial\widehat u_h}{\partial\xi}
(g_r,-1),
\end{align}
which gives
\begin{equation}\label{eq:vi0}
v_{i,0}^{K^h}
=
\widehat u_h(-1,-1)
+
\sum_{r=1}^{i}
A_r
\frac{\partial\widehat u_h}{\partial\xi}
(g_r,-1).
\end{equation}
Similarly, one has
\begin{equation}\label{eq:v0j}
v_{0,j}^{K^h}
=
\widehat u_h(-1,-1)
+
\sum_{s=1}^{j}
A_s
\frac{\partial\widehat u_h}{\partial\eta}
(-1,g_s).
\end{equation}
In addition, formulations \eqref{pro} also yields
\begin{align}
\sum_{r=1}^{i}\sum_{s=1}^{j}
\left(
v_{r,s}^{K^h}
+
v_{r-1,s-1}^{K^h}
-
v_{r-1,s}^{K^h}
-
v_{r,s-1}^{K^h}
\right)=\sum_{r=1}^{i}\sum_{s=1}^{j}
A_rA_s
\frac{\partial^2\widehat u_h}
{\partial\xi\partial\eta}
(g_r,g_s).
\label{eq:double-sum}
\end{align}
The left-hand side of Equation \eqref{eq:double-sum} is actually
\begin{align}
\sum_{r=1}^{i}\sum_{s=1}^{j}
\left(
v_{r,s}^{K^h}
+
v_{r-1,s-1}^{K^h}
-
v_{r-1,s}^{K^h}
-
v_{r,s-1}^{K^h}
\right)
=
v_{i,j}^{K^h}
-
v_{i,0}^{K^h}
-
v_{0,j}^{K^h}
+
v_{0,0}^{K^h}.\label{exen}
\end{align}
Therefore,
\begin{equation}\label{eq:vij-middle}
v_{i,j}^{K^h}
=
v_{i,0}^{K^h}
+
v_{0,j}^{K^h}
-
v_{0,0}^{K^h}
+
\sum_{r=1}^{i}\sum_{s=1}^{j}
A_rA_s
\frac{\partial^2\widehat u_h}
{\partial\xi\partial\eta}
(g_r,g_s).
\end{equation}
Substituting \eqref{eq:vi0} and \eqref{eq:v0j} into
\eqref{eq:vij-middle}, we obtain the explicit formula
\begin{align}
v_{i,j}^{K^h}
={}&
\widehat u_h(-1,-1)
+
\sum_{r=1}^{i}
A_r
\frac{\partial\widehat u_h}{\partial\xi}
(g_r,-1)
\nonumber\\
&+
\sum_{s=1}^{j}
A_s
\frac{\partial\widehat u_h}{\partial\eta}
(-1,g_s)
\nonumber\\
&+
\sum_{r=1}^{i}\sum_{s=1}^{j}
A_rA_s
\frac{\partial^2\widehat u_h}
{\partial\xi\partial\eta}
(g_r,g_s).
\label{eq23}
\end{align}

We next prove uniqueness. Suppose that
$\{v_{i,j}^{K^h}\}$ and
$\{\widetilde v_{i,j}^{K^h}\}$ are two sets of coefficients
satisfying \eqref{pro} for the same $u_h$. Define
$
w_{i,j}^{K^h}
=
v_{i,j}^{K^h}
-
\widetilde v_{i,j}^{K^h}.
$
The equations \eqref{pro} imply
\[
w_{0,0}^{K^h}=0,\qquad
w_{i,0}^{K^h}-w_{i-1,0}^{K^h}=0,
\qquad
w_{0,j}^{K^h}-w_{0,j-1}^{K^h}=0,
\]
and
\[
w_{i,j}^{K^h}
+
w_{i-1,j-1}^{K^h}
-
w_{i-1,j}^{K^h}
-
w_{i,j-1}^{K^h}
=0.
\]
The first three relations imply
\[
w_{i,0}^{K^h}=0,
\qquad
w_{0,j}^{K^h}=0.
\]
The last relation can be written as
\[
w_{i,j}^{K^h}
=
w_{i-1,j}^{K^h}
+
w_{i,j-1}^{K^h}
-
w_{i-1,j-1}^{K^h},
\]
which gives
\[
w_{i,j}^{K^h}=0,
\qquad
0\le i,j\le k.
\]
Thus the local coefficients are unique.

\end{proof}


For $q\in P_k([-1,1])$, we denote the one-dimensional Gaussian transfer operator by $\mathcal Q_k:P_k([-1,1])
\longrightarrow
P_0(\{[g_i,g_{i+1}]\}_{i=0}^{k})
$ with
$
(\mathcal Q_kq)|_{[g_i,g_{i+1}]}
=
q_i^*.
$
Here the coefficients $q^*_i$ satisfy
\begin{equation}\label{Tk-def}
q_0^*=q(-1),
\qquad
q_i^*=q_{i-1}^*+
A_iq'(g_i),
\qquad
1\le i\le k.
\end{equation}
It follows from \eqref{Tk-def} that we have $q^*_k=q(1)$. For any $p,q\in P_k([-1,1])$, define
\begin{equation}\label{Bk-def}
B_k(p,q)=
\int_{-1}^{1}
p(t)\mathcal Q_kq(t)
\,{\rm d}t.
\end{equation}

\begin{lemma}\label{lemma:gaussian-transfer}
For any $p(t),q(t)\in P_k([-1,1])$ with expressions
\[
p(t)=\sum_{\ell=0}^{k}p_\ell t^\ell,
\qquad
q(t)=\sum_{\ell=0}^{k}q_\ell t^\ell,
\]
 there exists a positive constant $c_k$ depending on the degree $k$, such that
\begin{equation}\label{Bk-identity}
B_k(p,q)
=
\int_{-1}^{1}p(t)q(t)\,{\rm d}t
+
c_kp_kq_k.
\end{equation}
Moreover, we also have
\begin{align}
&B_k(q,q)
\ge
\|q\|_{0,(-1,1)}^2,\label{Bk-positive}\\
&\|\mathcal Q_kq\|_{0,(-1,1)}
\le
C\|q\|_{0,(-1,1)}\label{Tk-stability}.
\end{align}
\end{lemma}

\begin{proof}
Let
\[
P(t)
=
\int_{-1}^{t}p(s)\,{\rm d}s,
\]
then
\begin{align*}
B_k(p,q)
=
\sum_{i=0}^{k}
q_i^*
\int_{g_i}^{g_{i+1}}
p(t)\,{\rm d}t
=
\sum_{i=0}^{k}
q_i^*
\left(
P(g_{i+1})-P(g_i)
\right).
\end{align*}
A direct calculation yields
\begin{align*}
B_k(p,q)
=
q_k^*P(1)
-
\sum_{i=1}^{k}
(q_i^*-q_{i-1}^*)P(g_i)=
q(1)P(1)
-
\sum_{i=1}^{k}
A_iq'(g_i)P(g_i).
\end{align*}
By integration by parts, we have
\[
\int_{-1}^{1}p(t)q(t)\,{\rm d}t
=
q(1)P(1)
-
\int_{-1}^{1}q'(t)P(t)\,{\rm d}t.
\]
Consequently,
\begin{align}
B_k(p,q)
-
\int_{-1}^{1}p(t)q(t)\,{\rm d}t
&=
\int_{-1}^{1}q'(t)P(t)\,{\rm d}t
-
\sum_{i=1}^{k}
A_iq'(g_i)P(g_i)\nonumber\\
&=c_kp_kq_k,
\label{Bk-error}
\end{align}
where $c_k=\frac{k}{k+1}$ is the coefficient of the $2k$th-order Gaussian remainder term. This proves \eqref{Bk-identity}. Taking $p=q$ in \eqref{Bk-error} gives \eqref{Bk-positive}.

Finally, by \eqref{Tk-def} and the equivalence of norms in finite-dimensional space, we can obtain
\[
|q_i^*|
\leq \|q\|_{\infty,(-1,1)}\leq C\|q\|_{0,(-1,1)},
\]
thus
\[
\|\mathcal Q_kq\|_{0,(-1,1)}^2
=
\sum_{i=0}^{k}
(g_{i+1}-g_i)|q_i^*|^2
\le
C\|q\|_{0,(-1,1)}^2.
\]
\end{proof}


\begin{lemma}\label{lemma:Pi_star_bound}
Assume that $\mathcal T_h$ is a regular curved-edge mesh partition
of $\Omega_h$. For any $u_h\in\mathcal U^k_h$, we have
\begin{equation}\label{robin-Pi-full}
\|
\Pi_{h}^{k,*}u_h
\|_{\mathcal T_h^*}
\le
C
|||u_h|||_{\Omega_h}.
\end{equation}
\end{lemma}

\begin{proof}
Denote $v_h=\Pi_{h}^{k,*}u_h$. We first derive the representation of the jump
$[v_h]_{i,j}^{K^h,x}$. For fixed $j\in\mathbb Z_k$, using
\eqref{id1} and \eqref{pro}, we have
\begin{align}
[v_h]_{i,j}^{K^h,x}
-
[v_h]_{i-1,j}^{K^h,x}
&=
\lfloor v_h\rfloor_{i,j}^{K^h}=A_iA_j
\frac{\partial^2\widehat u_h}
{\partial\xi\partial\eta}
(g_i,g_j),
\qquad
i\in\mathbb Z_k.
\label{jump-x-recursion}
\end{align}
It follows from \eqref{pro} that one has
\begin{equation}\label{jump-x-initial}
[v_h]_{0,j}^{K^h,x}=v^{K^h,x}_{0,j}-v^{K^h,x}_{0,j-1}
=
A_j
\frac{\partial\widehat u_h}{\partial\eta}
(-1,g_j),
\end{equation}
which together with \eqref{jump-x-recursion} gives
\begin{align}
[v_h]_{i,j}^{K^h,x}
&=
[v_h]_{0,j}^{K^h,x}
+
\sum_{r=1}^{i}
\left(
[v_h]_{r,j}^{K^h,x}
-
[v_h]_{r-1,j}^{K^h,x}
\right)
\nonumber\\
&=
A_j
\frac{\partial\widehat u_h}{\partial\eta}
(-1,g_j)
+
A_j
\sum_{r=1}^{i}
A_r
\frac{\partial^2\widehat u_h}
{\partial\xi\partial\eta}
(g_r,g_j)
\nonumber\\
&=A_j\big(q_j(-1)+\sum_{r=1}^{i}A_rq_j'(g_r)\big),
\label{jump-x-explicit}
\end{align}
where $q_j(\xi)=\frac{\partial\widehat u_h}{\partial\eta}(\xi,g_j)\in P_k([-1,1])$. Using $(\mathcal Q_kq_j)_0=q_j(-1)$ and
$(\mathcal Q_kq_j)_i-(\mathcal Q_kq_j)_{i-1}=A_iq_j'(g_i)$, we have
\begin{equation}\label{Tk-explicit-x}
(\mathcal Q_kq_j)_i
=
q_j(-1)
+
\sum_{r=1}^{i}
A_rq_j'(g_r).
\end{equation}
Combining \eqref{jump-x-explicit} and \eqref{Tk-explicit-x}, one has
\begin{equation}\label{jump-Tx}
[v_h]_{i,j}^{K^h,x}
=
A_j
\left(
\mathcal Q_k
\big(
\frac{\partial\widehat u_h}
{\partial\eta}(\cdot,g_j)
\big)
\right)_i .
\end{equation}
Similarly, we obtain
\begin{equation}\label{jump-Ty}
[v_h]_{i,j}^{K^h,y}
=
A_i
\left(
\mathcal Q_k
\big(
\frac{\partial\widehat u_h}
{\partial\xi}(g_i,\cdot)
\big)
\right)_j .
\end{equation}
By the shape regularity of the mesh and \eqref{Tk-stability}, we have
\begin{align}
|v_h|_{\mathcal T_h^*}^2
&\le C
\sum_{K^h\in\mathcal T_h}
\bigg(
\sum_{j=1}^{k}\sum_{i=0}^{k}
|[v_h]_{i,j}^{K^h,x}|^2+
\sum_{i=1}^{k}\sum_{j=0}^{k}
|[v_h]_{i,j}^{K^h,y}|^2
\bigg)\nonumber\\
&\leq C
\sum_{K^h\in\mathcal T_h}\bigg(
\sum_{j=1}^{k}
|A_j|\|
\frac{\partial\widehat u_h}
{\partial\eta}(\cdot,g_j)
\|_{0,(-1,1)}^2+
\sum_{i=1}^{k}
|A_i|\|
\frac{\partial\widehat u_h}
{\partial\xi}(g_i,\cdot)
\|_{0,(-1,1)}^2\bigg)\nonumber\\
&\leq C\sum_{K^h\in\mathcal T_h}|\widehat{u}_h|^2_{1,\widehat{K}}\leq C|u_h|^2_{1,\Omega_h}.\label{ele}
\end{align}
A direct calculation yields
\begin{equation}\label{eq17}
\|
\sigma_h^{1/2}
\Pi_{h}^{k,*}u_h
\|_{0,\Gamma_{h}}
\le
C
\|
\sigma_h^{1/2}u_h
\|_{0,\Gamma_{h}},
\end{equation}
which together with \eqref{ele} gives
\eqref{robin-Pi-full}.
\end{proof}


For a matrix-valued function
$
\mathbb D
=
\begin{pmatrix}
d_{11}&d_{12}\\
d_{21}&d_{22}
\end{pmatrix},
$
we  define
\begin{align}
a_{K^h,\mathbb D}(u_h,v_h)
=&
\sum_{i=0}^{k}
\sum_{j=1}^{k}
[v_h]_{i,j}^{K^h,x}
\int_{g_i}^{g_{i+1}}
\left(
d_{21}
\frac{\partial\widehat u_h}{\partial\xi}
+
d_{22}
\frac{\partial\widehat u_h}{\partial\eta}
\right)(\xi,g_j)
\,{\rm d}\xi
\nonumber\\
&+
\sum_{i=1}^{k}
\sum_{j=0}^{k}
[v_h]_{i,j}^{K^h,y}
\int_{g_j}^{g_{j+1}}
\left(
d_{11}
\frac{\partial\widehat u_h}{\partial\xi}
+
d_{12}
\frac{\partial\widehat u_h}{\partial\eta}
\right)(g_i,\eta)
\,{\rm d}\eta.
\label{aKD-def}
\end{align}
By the normal derivative transformation \eqref{eq4} and \eqref{eq5}, we have
\begin{equation}\label{ah0K-aKD}
a_{h,0,K^h}(u_h,v_h)=a_{K^h,\mathbb D_K}(u_h,v_h),
\end{equation}
where the matrix $\mathbb D_K$ with  $\widehat \kappa_K=\kappa\circ\Psi_K$ is
\begin{equation}\label{DK-def}
\mathbb D_K(\xi,\eta)=
\frac{\widehat \kappa_K(\xi,\eta)}
{J_K(\xi,\eta)}
\begin{pmatrix}
r_\eta^2&-s\\
-s&r_\xi^2
\end{pmatrix}.
\end{equation}


\begin{theorem}\label{lemma4}
Suppose that $\Omega_h$ is a shape-regular curved quadrilateral mesh
approximation of $\Omega$. Then, for sufficiently small $h$, there exists a constant $C>0$
independent of $h$, such that
\begin{equation}\label{co1}
a_h(u_h,\Pi_{h}^{k,*}u_h)\geq C|||u_h|||^2_{\Omega_h},
\qquad
\forall u_h\in\mathcal U_h^k.
\end{equation}
\end{theorem}

\begin{proof}
We will prove this lemma in the following three steps.

\medskip
\noindent
{\bf Step 1 (Properties of the transformed  matrix).}
Using the $W^{1,\infty}$ Poincar\'e inequality, we have
\begin{align}
&\|\mathbb D_K-\overline {\mathbb D_K}\|_{\infty,\widehat K}\nonumber\\
\leq& C\|D\mathbb D_K\|_{\infty,\widehat K}\nonumber\\
\leq& C(\big\|
D(r_\eta^2J^{-1}_K)
\big\|_{\infty,\widehat K}
+
\big\|
D(sJ^{-1}_K)
\big\|_{\infty,\widehat K}
+
\big\|
D(r_\xi^2J^{-1}_K)
\big\|_{\infty,\widehat K}),\label{DKK-osc}
\end{align}
which together with \eqref{eq:metric_coeff_est} yields
\begin{equation}\label{DK-osc}
\|\mathbb D_K-\overline {\mathbb D_K}\|_{\infty,\widehat K}\leq Ch.
\end{equation}
We recall that, by the uniform non-degeneracy of $\Psi_K$ and
$\kappa\ge a\kappa_0>0$, there exist constants $c_0,C_0>0$, independent of $h$, such that
\begin{equation}\label{DK-positive-proof}
c_0|\zeta|^2
\le
\zeta^T\mathbb D_K(\xi,\eta)\zeta
\le
C_0|\zeta|^2,
\qquad
\forall\zeta\in\mathbb R^2.
\end{equation}
Averaging \eqref{DK-positive-proof} over $\widehat K$ yields
\begin{equation}\label{DKbar-positive-proof}
c_0|\zeta|^2
\le
\zeta^T\overline{\mathbb D_K}\zeta
\le
C_0|\zeta|^2.
\end{equation}
\medskip
\noindent
{\bf Step 2 (Positivity of the $a_{h,0,K^h}(\cdot,\cdot))$.} For simplicity, we write
$\widehat u_\xi=\frac{\partial\widehat u_h}{\partial\xi}$ and $\widehat u_\eta
=\frac{\partial\widehat u_h}{\partial\eta}$, and denote 
\begin{equation}\label{DKbar-proof-def}
\begin{pmatrix}
\overline d_{11} & \overline d_{12}\\
\overline d_{12} & \overline d_{22}
\end{pmatrix}=\overline{\mathbb D_K}=
\frac{1}{|\widehat K|}
\int_{\widehat K}
\mathbb D_K(\xi,\eta)\,{\rm d}\xi{\rm d}\eta.
\end{equation}
Substituting \eqref{jump-Tx} and
\eqref{jump-Ty} into $a_{K^h,\overline{\mathbb D_K}}$ and using the definition of $B_k$ give
\begin{align}
a_{K^h,\overline {\mathbb D_K}}
\big(
u_h,\Pi_{h}^{k,*}u_h
\big)={}&
\overline d_{22}
\sum_{j=1}^{k}
A_j
B_k
\left(
\widehat u_\eta(\cdot,g_j),
\widehat u_\eta(\cdot,g_j)
\right)
\nonumber\\
&+
\overline d_{11}
\sum_{i=1}^{k}
A_i
B_k
\left(
\widehat u_\xi(g_i,\cdot),
\widehat u_\xi(g_i,\cdot)
\right)
\nonumber\\
&+
\overline d_{12}
\sum_{j=1}^{k}
A_j
B_k
\left(
\widehat u_\xi(\cdot,g_j),
\widehat u_\eta(\cdot,g_j)
\right)
\nonumber\\
&+
\overline d_{12}
\sum_{i=1}^{k}
A_i
B_k
\left(
\widehat u_\eta(g_i,\cdot),
\widehat u_\xi(g_i,\cdot)
\right).
\label{frozen-expansion-proof}
\end{align}
On one hand, for fixed $g_j$ and $g_i$, since $\widehat u_\eta(\cdot,g_j),\widehat u_\xi(g_i,\cdot)\in P_k([-1,1])$ and using \eqref{Bk-identity}, we have
\begin{align}
&B_k
\left(
\widehat u_\eta(\cdot,g_j),
\widehat u_\eta(\cdot,g_j)
\right)=
\int_{-1}^{1}
\widehat u_\eta^2(\xi,g_j)\,{\rm d}\xi
+c_k(\widehat u_\eta(\cdot,g_j))^2,
\label{diag-eta}\\
&B_k
\left(
\widehat u_\xi(g_i,\cdot),
\widehat u_\xi(g_i,\cdot)
\right)=
\int_{-1}^{1}
\widehat u_\xi^2(g_i,\eta)\,{\rm d}\eta
+c_k(\widehat u_\xi(g_i,\cdot))^2,
\label{diag-xi}
\end{align}
On the other hand, since $\widehat u_\xi(\cdot,g_j),\widehat u_\eta(g_i,\cdot)\in P_{k-1}([-1,1])$,  from
\eqref{Bk-identity}, one has
\begin{align}
B_k
\left(
\widehat u_\xi(\cdot,g_j),
\widehat u_\eta(\cdot,g_j)
\right)
&=
\int_{-1}^{1}
\widehat u_\xi(\xi,g_j)
\widehat u_\eta(\xi,g_j)
\,{\rm d}\xi,\label{mixed-1}\\
B_k
\left(
\widehat u_\eta(g_i,\cdot),
\widehat u_\xi(g_i,\cdot)
\right)
&=
\int_{-1}^{1}
\widehat u_\eta(g_i,\eta)
\widehat u_\xi(g_i,\eta)
\,{\rm d}\eta.\label{mixed-2}
\end{align}
Substituting
\eqref{diag-eta}-\eqref{mixed-2}
into \eqref{frozen-expansion-proof}, we obtain
\begin{align}
\nonumber a_{K^h,\overline {\mathbb D_K}}\big(u_h,\Pi_{h}^{k,*}u_h\big)={}&\overline d_{22}\sum_{j=1}^{k}A_j\Big(\int_{-1}^{1}
\widehat u_\eta^2(\xi,g_j)\,{\rm d}\xi
+
c_k(\widehat u_\eta(\cdot,g_j))^2\Big)\\
\nonumber&+\overline d_{11}\sum_{i=1}^{k}A_i\Big(\int_{-1}^{1}
\widehat u_\xi^2(g_i,\eta)\,{\rm d}\eta
+
c_k(\widehat u_\xi(g_i,\cdot))^2\Big)\\
\nonumber&+\overline d_{12}\sum_{j=1}^{k}A_j\Big(\int_{-1}^{1}
\widehat u_\xi(\xi,g_j)
\widehat u_\eta(\xi,g_j)
\,{\rm d}\xi\Big)\\
&+\overline d_{12}\sum_{i=1}^{k}A_i\Big(\int_{-1}^{1}
\widehat u_\eta(g_i,\eta)
\widehat u_\xi(g_i,\eta)
\,{\rm d}\xi\Big).\label{pp5}
\end{align}
Since the $k$-point Gaussian integral is exact for $\widehat u_\eta^2(\xi,\eta)\in P_{2k-2}([-1,1])$ (with respect to $\eta$), one has
\begin{align}
\sum_{j=1}^{k}
A_j
\int_{-1}^{1}
\widehat u_\eta^2(\xi,g_j)\,{\rm d}\xi
&=
\int_{-1}^{1}\int_{-1}^{1}
\widehat u_\eta^2(\xi,\eta)
\,{\rm d}\xi{\rm d}\eta=\int_{\widehat K}
\widehat u_\eta^2(\xi,\eta)
\,{\rm d}\xi{\rm d}\eta.
\label{outer-eta}
\end{align}
Similarly, we have
\begin{align}
&\sum_{i=1}^{k}
A_i
\int_{-1}^{1}
\widehat u_\xi^2(g_i,\eta)\,{\rm d}\eta
=
\int_{\widehat K}
\widehat u_\xi^2(\xi,\eta)
\,{\rm d}\xi{\rm d}\eta.
\label{outer-xi}\\
&\sum_{j=1}^{k}
A_j
\int_{-1}^{1}
\widehat u_\xi(\xi,g_j)
\widehat u_\eta(\xi,g_j)
\,{\rm d}\xi
=
\int_{\widehat K}
\widehat u_\xi\widehat u_\eta
\,{\rm d}\xi{\rm d}\eta,
\label{outer-mixed1}\\
&\sum_{i=1}^{k}
A_i
\int_{-1}^{1}
\widehat u_\eta(g_i,\eta)
\widehat u_\xi(g_i,\eta)
\,{\rm d}\eta
=
\int_{\widehat K}
\widehat u_\eta\widehat u_\xi
\,{\rm d}\xi{\rm d}\eta.
\label{outer-mixed2}
\end{align}
Substituting
\eqref{outer-eta}--\eqref{outer-mixed2}
into \eqref{pp5} yields
\begin{align}
a_{K^h,\overline {\mathbb D_K}}
\left(
u_h,\Pi_{h}^{k,*}u_h
\right)
=
\int_{\widehat K}
\left[
\overline d_{11}\widehat u_\xi^2
+
2\overline d_{12}
\widehat u_\xi\widehat u_\eta
+
\overline d_{22}\widehat u_\eta^2
\right]
\,{\rm d}\xi{\rm d}\eta
+
\mathcal R_K,
\label{eq60-pre}
\end{align}
where
\begin{align*}
\mathcal R_K
=
\overline d_{22}
\sum_{j=1}^{k}
A_j
c_k(\widehat u_\eta(\cdot,g_j))^2+\overline d_{11}
\sum_{i=1}^{k}
A_i
c_k(\widehat u_\xi(g_i,\cdot))^2\geq0.
\end{align*}
By \eqref{DKbar-positive-proof}, we obtain
\begin{align}\label{pp6}
a_{K^h,\overline{\mathbb D_K}}
\left(
u_h,\Pi_{h}^{k,*}u_h
\right)&=
\int_{\widehat K}
(\widehat\nabla\widehat u_h)^T
\overline{\mathbb D_K}
\widehat\nabla\widehat u_h
\,{\rm d}\xi{\rm d}\eta\geq c_0
|\widehat u_h|_{1,\widehat K}^2.
\end{align}

We set
\[
E_K
:=
\mathbb D_K-\overline {\mathbb D_K}
=
\begin{pmatrix}
e_{11}&e_{12}\\
e_{21}&e_{22}
\end{pmatrix}.
\]
The fact $
a_{h,0,K^h}(u_h,v_h)
=
a_{K^h,\mathbb D_K}(u_h,v_h)$ gives
\begin{align}
a_{h,0,K^h}
\left(
u_h,\Pi_{h}^{k,*}u_h
\right)
-
a_{K^h,\overline D_K}
\left(
u_h,\Pi_{h}^{k,*}u_h
\right)
=
a_{K^h,E_K}
\left(
u_h,\Pi_{h}^{k,*}u_h
\right).
\label{perturb-identity}
\end{align}
Using
\eqref{frozen-expansion-proof} yields
\begin{align*}
&
\left|
a_{K,E_K}
\left(
u_h,\Pi_{h}^{k,*}u_h
\right)
\right|
\\
\le{}&
\|E_K\|_{\infty,\widehat K}
\sum_{j=1}^{k}
A_j
\int_{-1}^{1}
\left|
\mathcal Q_k
\big(
\widehat u_\eta(\cdot,g_j)
\big)
\right|
\left(
|\widehat u_\xi(\xi,g_j)|
+
|\widehat u_\eta(\xi,g_j)|
\right)
\,{\rm d}\xi
\\
&+
\|E_K\|_{\infty,\widehat K}
\sum_{i=1}^{k}
A_i
\int_{-1}^{1}
\left|
\mathcal Q_k
\big(
\widehat u_\xi(g_i,\cdot)
\big)
\right|
\left(
|\widehat u_\xi(g_i,\eta)|
+
|\widehat u_\eta(g_i,\eta)|
\right)
\,{\rm d}\eta\\
={}&I_x+I_y.
\end{align*}
For $I_x$, it follows from Cauchy--Schwarz inequality and \eqref{Tk-stability} that one has
\begin{align*}
I_x
\le{}&
C\|E_K\|_{\infty,\widehat K}
\sum_{j=1}^{k}
A_j
\left\|
\mathcal Q_k
\big(
\widehat u_\eta(\cdot,g_j)
\big)
\right\|_{0,(-1,1)}
\\
&\qquad\times
\left(
\|\widehat u_\xi(\cdot,g_j)\|_{0,(-1,1)}
+
\|\widehat u_\eta(\cdot,g_j)\|_{0,(-1,1)}
\right)\\
\le{}&
C\|E_K\|_{\infty,\widehat K}
\left(
\sum_{j=1}^{k}
A_j
\|\widehat u_\eta(\cdot,g_j)\|_0^2
\right)^{1/2}
\\
&\quad\times
\left(
\sum_{j=1}^{k}
A_j
\left(
\|\widehat u_\xi(\cdot,g_j)\|_0
+
\|\widehat u_\eta(\cdot,g_j)\|_0
\right)^2
\right)^{1/2}\\
\leq{}& C\|E_K\|_{\infty,\widehat K}
|\widehat u_h|^2_{1,\widehat K}.
\end{align*}
The same argument for $I_y$ gives
\[
I_y
\le
C
\|E_K\|_{\infty,\widehat K}
|\widehat u_h|_{1,\widehat K}^2.
\]
Therefore, we have
\begin{equation}\label{eq62}
\left|
a_{K^h,E_K}
\left(
u_h,\Pi_{h}^{k,*}u_h
\right)
\right|
\le
C
\|\mathbb D_K-\overline{\mathbb D_K}\|_{\infty,\widehat K}
|\widehat u_h|_{1,\widehat K}^2\leq Ch|\widehat u_h|_{1,\widehat K}^2.
\end{equation}
Combining \eqref{pp6}, \eqref{perturb-identity} and \eqref{eq62}, one has
\begin{align*}
a_{h,0,K^h}
\left(
u_h,\Pi_{h}^{k,*}u_h
\right)
&\ge
a_{K^h,\overline {\mathbb D_K}}
\left(
u_h,\Pi_{h}^{k,*}u_h
\right)
-
Ch
|\widehat u_h|_{1,\widehat K}^2
\\
&\ge
(c_0-Ch)
|\widehat u_h|_{1,\widehat K}^2,
\end{align*}
which implies 
\begin{equation}\label{local-diff-coercive}
a_{h,0,K^h}
\left(
u_h,\Pi_{h}^{k,*}u_h
\right)
\ge
C
|u_h|_{1,K^h}^2
\end{equation}
for sufficiently small $h$. 

\medskip
\noindent
{\bf Step 3 (Positivity of the 
$r_h(\cdot,\cdot)$).} The boundary term can be written as
\begin{equation}\label{rh-edge-sum}
r_h(
u_h,\Pi_{h}^{k,*}u_h)
=
\sum_{e_h\in\Gamma_h}
\int_{e_h}
\sigma_h u_h
\Pi_{h}^{k,*}u_h
\,{\rm d}s:=\sum_{e_h\in\Gamma_h}r_{h,e}
(
u_h,\Pi_{h}^{k,*}u_h
).
\end{equation}
For each $e_h=\Psi_K(\widehat e)\in\mathcal E_h^b$, we obtain
\begin{align}
r_{h,e}(
u_h,\Pi_{h}^{k,*}u_h)&=
\int_{\widehat{e}}
\sigma_h\bigl(\Psi_K(t)\bigr)
u_h\bigl(\Psi_K(t)\bigr)
\Pi_{h}^{k,*}u_h
\bigl(\Psi_K(t)\bigr)
\rho_e(t)\,{\rm d}t\nonumber\\
&=\int_{\widehat{e}}
w_e(t)
q_e(t)
\mathcal Q_kq_e(t)
\,{\rm d}t
\label{rh-reference-edge}
\end{align}
where $\rho_e(t)=\left|\frac{{\rm d}}{{\rm d}t}\Psi_K(t)\right|$, $q_e(t)=u_h\bigl(\Psi_K(t)\bigr)$ and $w_e(t)=\sigma_h\bigl(\Psi_K(t)\bigr)\rho_e(t).$
By Lemma~\ref{lem:geometric_coefficient_bounds} and boundedness of $\sigma_h$, one has
\begin{equation}\label{we-osc-proof}
\|w_e-\overline w_e\|_{\infty,(-1,1)}
\le
Ch\overline w_e,
\end{equation}
where $\overline w_e$ is the mean value of $w_e$ in $[-1,1]$. Thus,  
\begin{align}
r_{h,e}(
u_h,\Pi_{h}^{k,*}u_h)&=
\overline w_e
\int_{\widehat{e}}
q_e\mathcal Q_kq_e\,{\rm d}t
+
\int_{\widehat{e}}
(w_e-\overline w_e)
q_e\mathcal Q_kq_e
\,{\rm d}t
\nonumber\\
&=
\overline w_e B_k(q_e,q_e)
+
\int_{\widehat{e}}
(w_e-\overline w_e)
q_e\mathcal Q_kq_e
\,{\rm d}t,
\label{robin-split-proof}
\end{align}
which together with Lemma~\ref{lemma:gaussian-transfer} gives
\begin{align*}
&r_{h,e}(
u_h,\Pi_{h}^{k,*}u_h)
\\
\ge&
C\big(\overline w_e
\|q_e\|_{0,(-1,1)}^2
-
\|w_e-\overline w_e\|_{\infty,(-1,1)}
\|q_e\|_{0,(-1,1)}
\|\mathcal Q_kq_e\|_{0,(-1,1)}\big)
\\
\ge&
(1-Ch)\overline w_e\|q_e\|_{0,(-1,1)}^2.
\end{align*}
Thus, for sufficiently small $h$, one has
\begin{equation}\label{rh-edge-positive}
r_{h,e}(
u_h,\Pi_{h}^{k,*}u_h)
\ge
C\overline w_e
\|q_e\|_{0,(-1,1)}^2,
\end{equation}
which implies
\begin{equation}\label{robin-coercive-proof}
r_h(
u_h,\Pi_{h}^{k,*}u_h)
\ge
C
\|\sigma_h^{1/2}u_h\|_{0,\Gamma_h}^2.
\end{equation}
Finally, combining \eqref{local-diff-coercive} and \eqref{robin-coercive-proof}, we complete the proof.
\end{proof}

\begin{theorem}\label{lemma42}
Under the assumptions of Theorem~\ref{lemma4}, there exists a positive
constant $C$ independent of $h$, such that
\begin{equation}\label{co2}
\inf_{0\ne u_h\in\mathcal U_h^k}
\sup_{0\ne v_h\in\mathcal V_h}
\frac{
a_h(u_h,v_h)
}{
|||u_h|||_{\Omega_h}
\|v_h\|_{\mathcal T_h^*}
}
\ge
C.
\end{equation}
\end{theorem}

\begin{proof}
By Lemma~\ref{lemma4}, we have
\[
a_h(
u_h,\Pi_{h}^{k,*}u_h)
\ge
C|||u_h|||^2_{\Omega_h}.
\]
Furthermore, Lemma~\ref{lemma:Pi_star_bound} yields
\[
\|
\Pi_{h}^{k,*}u_h
\|_{\mathcal T_h^*,R}
\le
C|||u_h|||_{\Omega_h}.
\]
Therefore, one has
\begin{align}
\sup_{0\ne v_h\in\mathcal V_h}
\frac{
a_h(u_h,v_h)
}{
\|v_h\|_{\mathcal T_h^*}
}
\ge
\frac{
a_h
\left(
u_h,\Pi_{h}^{k,*}u_h
\right)
}{
\|
\Pi_{h}^{k,*}u_h
\|_{\mathcal T_h^*}
}\ge
C
|||u_h|||_{\Omega_h}.
\label{ineq1}
\end{align}
Dividing \eqref{ineq1} by $|||u_h|||_{\Omega_h}$ and taking the infimum over
all $0\ne u_h\in\mathcal U_h^k$, we obtain
\[
\inf_{0\ne u_h\in\mathcal U_h^k}
\sup_{0\ne v_h\in\mathcal V_h}
\frac{
a_h(u_h,v_h)
}{
|||u_h|||_{\Omega_h}
\|v_h\|_{\mathcal T_h^*}
}
\ge
C.
\]
This completes the proof.
\end{proof}

\subsection{Boundedness}

To end this section, we establish the boundedness of the
bilinear form \eqref{robin-ah-decomp}.

\begin{theorem}\label{thm:boundedness}
The bilinear form $a_h(\cdot,\cdot)$ is bounded. More precisely,
for any $u_h\in\mathcal U_h^k$ and $v_h\in\mathcal V_h$, we have
\begin{equation}\label{bou}
|a_h(u_h,v_h)|
\le
C
|||u_h|||_{\Omega_h}
\|v_h\|_{\mathcal T_h^*}.
\end{equation}
\end{theorem}

\begin{proof}
Recall that
\[
a_h(u_h,v_h)
=
a_{h,0}(u_h,v_h)
+
r_h(u_h,v_h).
\]
Using the Cauchy-Schwarz inequality, we have
\begin{align}
|a_{h,0}(u_h,v_h)|
=
\Big|\sum_{e\in\mathcal E_{\mathcal T_h^*}}
[v_h]_e
\int_e
\kappa_h(x,y)
\frac{\partial u_h}{\partial\bm n_e}
\,{\rm d}s\Big|\leq C
|u_h|_{1,\Omega_h}
|v_h|_{\mathcal T_h^*}\label{bou1}
\end{align}
and
\begin{equation}\label{bou2}
|r_h(u_h,v_h)|
=
\Big|\sum_{P\in\mathcal N^b}
v_{K_P^{*,h}}
\int_{\Gamma_{P,h}}
\sigma_hu_h\,{\rm d}s\Big|\leq C
\|\sigma_h^{1/2}u_h\|_{0,\Gamma_h}
\|\sigma_h^{1/2}v_h\|_{0,\Gamma_h}.
\end{equation}
Thus,
\begin{align*}
|a_h(u_h,v_h)|
&\le
C
|u_h|_{1,\Omega_h}
|v_h|_{\mathcal T_h^*}+
\|\sigma_h^{1/2}u_h\|_{0,\Gamma_h}
\|\sigma_h^{1/2}v_h\|_{0,\Gamma_h}
\\
&\le
C
\left(
|u_h|_{1,\Omega_h}^2
+
\|\sigma_h^{1/2}u_h\|_{0,\Gamma_h}^2
\right)^{1/2}
\left(
|v_h|_{\mathcal T_h^*}^2
+
\|\sigma_h^{1/2}v_h\|_{0,\Gamma_h}^2
\right)^{1/2}
\\
&=
C|||u_h|||_{\Omega_h}
\|v_h\|_{\mathcal T_h^*}.
\end{align*}
\end{proof}

\section{Error Analysis}

In this section, we establish the error estimates in the \(H^1\)-norm
and \(L^2\)-norm. Since each curved-edge element \(K\) may not
coincide with its geometric approximation \(K^h\), the geometric error
has to be taken into account. For clarity, we state below the geometric properties of the boundary
approximation that will be used in the subsequent analysis.

\begin{assumption}\label{ass_geo}
For each boundary element $K\in\mathcal T$, let
$S_K=(K^h\setminus K)\cup(K\setminus K^h)$
denote the local geometric mismatch region. We assume that $S_K$
can be covered by a uniformly bounded number of open rectangles $\mathcal R_K
=\{R_1,R_2,\ldots,R_{n(K)}\}$.
For each $R\in\mathcal R_K$, it contains a portion of
$\partial\Omega\cap\partial K$ and a portion of
$\partial\Omega_h\cap\partial K^h$, and is defined by a translation and rotation of the $(0,L_R)\times(0,H_R)$ with $\max\{L_R,H_R\}\le Ch$.
For every $K\in\mathcal T$ and every $R\in\mathcal R_K$, the exact and approximate boundary
portions admit  parametrizations
\[
\gamma_{K,R}:[-1,1]
\rightarrow
\partial\Omega\cap\partial K\cap R,
\]
and
\[
\gamma_{K,R}^h:[-1,1]
\rightarrow
\partial\Omega_h\cap\partial K^h\cap R,
\]
with the same orientation. Let $\bm n_{K,R}$ and $\bm n_{K,R}^h$ be the unit normal vectors
corresponding to the prescribed common orientation, and define 
\[
N_{K,R}(\xi)=
\bm n_{K,R}(\xi)|\gamma_{K,R}'(\xi)|,
\qquad
N_{K,R}^h(\xi)=
\bm n_{K,R}^h(\xi)|(\gamma_{K,R}^h)'(\xi)|.
\]
We assume that, for some $\alpha\ge2$,
\begin{equation}\label{geo-position}
\|\gamma_{K,R}^h-\gamma_{K,R}\|_{\infty,(-1,1)}
\le
Ch^\alpha,
\end{equation}
and
\begin{equation}\label{geo-tangent}
\|N^h_{K,R}-N_{K,R}\|_{\infty,(-1,1)}
\le
Ch^\alpha.
\end{equation}
Moreover, the parametrizations are uniformly regular:
\begin{equation}\label{geo-param-regularity}
ch
\le
|\gamma_{K,R}'(\xi)|,\,|(\gamma_{K,R}^h)'(\xi)|
\le
Ch,
\qquad
\xi\in[-1,1].
\end{equation}
\end{assumption}
In this assumption, the parameter $\alpha$ is determined by the degree of the polynomial
used for boundary interpolation. For example, when a sufficiently
smooth boundary is approximated by a polynomial of degree $q$, we have $\alpha=q+1$. Numerical example 6.3 illustrates this relationship intuitively.
Furthermore, the assumption implies that each line passing through the height of the rectangle
intersects each boundary exactly once. Therefore, it does not include the case shown in Figure \ref{fig:dd}(b).


\begin{figure}[htbp]
    \centering
    \begin{subfigure}[b]{0.45\textwidth}
        \centering
        \includegraphics[width=\linewidth]{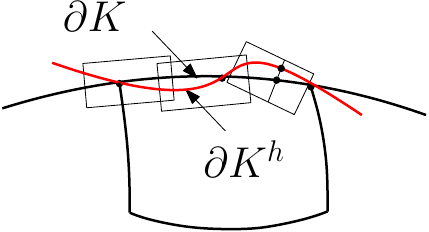}
        \caption{}
    \end{subfigure}
    \hfill
    \begin{subfigure}[b]{0.45\textwidth}
        \centering
        \includegraphics[width=\linewidth]{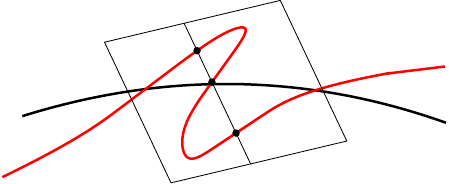}
        \caption{}
    \end{subfigure}
    \vspace{0.5em}
    \caption{Parts of boundary $\partial K$ (black line) and $\partial K^h $(red line). (a) Case satisfying Assumption \ref{ass_geo};  (b) Case not satisfying Assumption \ref{ass_geo} }
    \label{fig:dd}
\end{figure}


\subsection{$H^1$ error analysis}
\begin{theorem}\label{theorem4}
Let $\Omega_h$ be a regular approximation of $\Omega$. Assume that
$\Gamma=\partial\Omega$ is piecewise smooth, and that the
geometric approximation satisfies Assumption~\ref{ass_geo}. Let $U_\Gamma$ be a fixed open neighborhood of $\Gamma$ independent of
$h$ satisfying $\overline{\Omega\cup U_\Gamma}\subset \widetilde\Omega.$
For sufficiently small $h$, we assume that $\Omega_h\subset \widetilde\Omega$ and $(\Omega\setminus\Omega_h)\cup(\Omega_h\setminus\Omega)\subset U_\Gamma$.
Let $u\in H^{k+1}(\Omega)$ and $u_h\in\mathcal U_h^k$ be the solutions of
\eqref{model} and \eqref{scheme} respectively, and define
\[
F:=\widetilde\kappa\nabla\widetilde u,
\qquad
H:=\widetilde\sigma\,\widetilde u-\widetilde g,
\]
where $\widetilde u$ is a
bounded extension of $u$ to $\widetilde\Omega$. If $F\in H^2(U_\Gamma)^2$ and $H\in H^2(U_\Gamma)$, we have
\begin{equation}\label{eqq28}
\|\widetilde u-u_h\|_{1,\Omega_h}
\le
C\left(
h^k\|u\|_{k+1,\Omega}
+
h^{\alpha-1}\mathcal M_\Gamma
\right),
\end{equation}
where $\mathcal M_\Gamma:=\|F\|_{2,U_\Gamma}+\|H\|_{2,U_\Gamma}$.
\end{theorem}
\begin{proof}
Let $u_I=\Pi_h^k\widetilde u\in\mathcal U_h^k.$ Since $u_h$ satisfies
\begin{equation}\label{ex2}
a_h(u_h,v_h)
=
\ell_h(v_h),
\qquad
\forall v_h\in\mathcal V_h,
\end{equation}
where
\begin{align*}
\ell_h(v_h)=
\sum_{P\in\mathcal N}
v_{K_P^{*,h}}
\int_{K_P^{*,h}}
f_h\,{\rm d}x{\rm d}y
+
\sum_{P\in\mathcal N^b}
v_{K_P^{*,h}}
\int_{\Gamma_{P,h}}
g_h\,{\rm d}s,
\end{align*}
we have
\begin{align}
a_h(u_h-u_I,v_h)=
a_h(\widetilde u-u_I,v_h)
+
\mathcal G_h^R(\widetilde u,v_h),
\label{qq}
\end{align}
where geometric error
\begin{equation}\label{GR-def}
\mathcal G_h^R(\widetilde u,v_h)=
\ell_h(v_h)-a_h(\widetilde u,v_h).
\end{equation}
Recall that
\begin{equation}\label{coup}
a_h(\widetilde u-u_I,v_h)
=
a_{h,0}(\widetilde u-u_I,v_h)+r_h(\widetilde u-u_I,v_h).
\end{equation}
By the Cauchy--Schwarz inequality, the trace
inequality and the interpolation estimate \eqref{eq:interp_error}, we can obtain
\begin{align}
|a_{h,0}(\widetilde u-u_I,v_h)|
&\le
C
\Big(
\sum_{K^h\in\mathcal T_h}
\bigl(
|\widetilde u-u_I|_{1,K^h}^2
+
h_K^2|\widetilde u-u_I|_{2,K^h}^2
\bigr)
\Big)^{1/2}
|v_h|_{\mathcal T_h^*}\nonumber\\
&\leq C h^k
\|\widetilde u\|_{k+1,\Omega_h}
|v_h|_{\mathcal T_h^*}
\label{interp_flux_est2}
\end{align}
and
\begin{align}
|r_h(\widetilde u-u_I,v_h)|\le
C
\|\widetilde u-u_I\|_{0,\Gamma_h}
\|v_h\|_{0,\Gamma_h}\le
C h^k
\|\widetilde u\|_{k+1,\Omega_h}
\|v_h\|_{\mathcal T_h^*}.
\label{interp_robin_est}
\end{align}
Thus, we combine \eqref{coup}-\eqref{interp_robin_est} to give
\begin{equation}\label{eqq87}
|a_h(\widetilde u-u_I,v_h)|
\le
C h^k
\|\widetilde u\|_{k+1,\Omega_h}
\|v_h\|_{\mathcal T_h^*}
\le
C h^k
\|u\|_{k+1,\Omega}
\|v_h\|_{\mathcal T_h^*}.
\tag{87}
\end{equation}
It remains to estimate the geometric error
$\mathcal G_h^R$. For each $P\in\mathcal N^b$, set
\[
D_P^+:=K_P^{*,h}\setminus\Omega.
\]
By the definitions of $a_h(\widetilde u,v_h)$ and $\ell_h(v_h)$, we have
\begin{align}
\mathcal G_h^R(\widetilde u,v_h)
={}&
\sum_{P\in\mathcal N}
v_{K_P^{*,h}}
\Big(
\int_{K_P^{*,h}} f_h\,{\rm d}x{\rm d}y
+
\int_{\partial_0K_P^{*,h}}
F\cdot\bm n_P\,{\rm d}s
\Big)
\nonumber\\
&\quad
-
\sum_{P\in\mathcal N^b}
v_{K_P^{*,h}}
\int_{\Gamma_{P,h}}
H\,{\rm d}s .
\label{eq:GR-expand}
\end{align}
Applying the divergence theorem to $K_P^{*,h}$ gives
\begin{equation}
\label{eq:divergence-control-volume}
\int_{\partial_0K_P^{*,h}}
F\cdot\bm n_P\,{\rm d}s
=
\int_{K_P^{*,h}}
\nabla\cdot F\,{\rm d}x{\rm d}y
-
\int_{\Gamma_{P,h}}
F\cdot\bm n^h\,{\rm d}s,
\end{equation}
which together with \eqref{eq:GR-expand} yields
\begin{align}
\mathcal G_h^R(\widetilde u,v_h)
={}&
\sum_{P\in\mathcal N}
v_{K_P^{*,h}}
\int_{K_P^{*,h}}
\left(
f_h+\nabla\cdot F
\right)
\,{\rm d}x{\rm d}y
\nonumber\\
&\quad
-
\sum_{P\in\mathcal N^b}
v_{K_P^{*,h}}
\int_{\Gamma_{P,h}}
\left(
F\cdot\bm n^h+H
\right)
\,{\rm d}s .
\label{eq:GR-residual-identity}
\end{align}
Let $\Gamma_P:=\Psi^{-1}_{K}(\Gamma_{P,h})$.
 Since $f_h=0$ in $\Omega_h\setminus\Omega$, $\nabla\cdot F=-f$ in $\Omega$ and $F\cdot\bm n+H=0$ on $\Gamma$, we can obtain
\begin{align}
\mathcal G_h^R(\widetilde u,v_h)
={}&
\sum_{P\in\mathcal N^b}
v_{K_P^{*,h}}\bigg(
\int_{D_P^+}
\nabla\cdot F\,{\rm d}x{\rm d}y
\nonumber\\
&\qquad
-
\Big(
\int_{\Gamma_{P,h}}(
F\cdot\bm n^h+H
){\rm d}s
-
\int_{\Gamma_P}(
F\cdot\bm n+H){\rm d}s\Big)
\bigg).
\label{local-GR}
\end{align}
Denote 
\[
S_h^+
:=
\bigcup_{P\in\mathcal N^b}D_P^+
=
\Omega_h\setminus\Omega.
\]
According to Assumption \ref{ass_geo}, it holds that
\begin{equation*}
|D_P^+|\le Ch^{\alpha+1},
\qquad
|S_h^+|\le Ch^\alpha.
\end{equation*}
Moreover, by the standard narrow-strip estimate \cite{gyc25,lmw10,lsl24}, for every
$w\in H^1(U_\Gamma)$, one has
\begin{equation}\label{strip-estimate}
\|w\|_{0,S_h^+}
\le
Ch^{\alpha/2}
\|w\|_{1,U_\Gamma}.
\end{equation}
It follows from the Cauchy-Schwarz inequality that 
\begin{align}
\Big|
\sum_{P\in\mathcal N^b}
v_{K_P^{*,h}}
\int_{D_P^+}
\nabla\cdot F\,{\rm d}x{\rm d}y
\Big|
=
\Big|
\int_{S_h^+}
v_h\,\nabla\cdot F\,{\rm d}x{\rm d}y
\Big|
&\le
\|v_h\|_{0,S_h^+}
\|\nabla\cdot F\|_{0,S_h^+}.
\label{DP-est-1}
\end{align}
Since $|D_P^+|\le Ch^\alpha|\Gamma_{P,h}|$ and
$v_h$ is constant on each boundary control volume, we have
\begin{align}
\|v_h\|_{0,S_h^+}^2
&=
\sum_{P\in\mathcal N^b}
|v_{K_P^{*,h}}|^2|D_P^+|
\nonumber\\
&\le
Ch^\alpha
\sum_{P\in\mathcal N^b}
|\Gamma_{P,h}|
|v_{K_P^{*,h}}|^2
\nonumber\\
&\le
Ch^\alpha
\|v_h\|_{0,\Gamma_h}^2
\nonumber\\
&\le
Ch^\alpha
\|v_h\|_{\mathcal T_h^*}^2.
\label{vh}
\end{align}
Combining \eqref{strip-estimate}, \eqref{DP-est-1} and \eqref{vh} yields
\begin{equation}\label{DP-est}
\Big|
\sum_{P\in\mathcal N^b}
v_{K_P^{*,h}}
\int_{D_P^+}
\nabla\cdot F\,{\rm d}x{\rm d}y
\Big|
\le
Ch^\alpha
\|\nabla\cdot F\|_{1,U_\Gamma}
\|v_h\|_{\mathcal T_h^*}.
\end{equation}
A direct calculation gives
\begin{align}
&
\int_{\Gamma_{P,h}}
\left(
F\cdot\bm n^h+H
\right){\rm d}s
-
\int_{\Gamma_P}
\left(
F\cdot\bm n+H
\right){\rm d}s
\nonumber\\
&=
\int_{-1}^{1}
\left[
F(\gamma_{K,R}^h(\xi))
\cdot
N_{K,R}^h(\xi)
-
F(\gamma_{K,R}(\xi))
\cdot
N_{K,R}(\xi)
\right]
{\rm d}\xi
\nonumber\\
&\quad+
\int_{-1}^{1}
\left[
H(\gamma_{K,R}^h(\xi))
J_{K,R}^h(\xi)
-
H(\gamma_{K,R}(\xi))
J_{K,R}(\xi)
\right]
{\rm d}\xi
\nonumber\\
&\triangleq M_1+M_2,
\label{boundary-param}
\end{align}
where
$J_{K,R}=|\gamma_{K,R}'|$ and
$J_{K,R}^h=|(\gamma_{K,R}^h)'|$. For $M_1$, It follows from the Cauchy--Schwarz inequality, \eqref{geo-tangent} and
\eqref{strip-estimate} that one has
\begin{align}
\left|
\sum_{P\in\mathcal N^b}
v_{K_P^{*,h}}M_1
\right|
&=
\left|
\sum_{P\in\mathcal N^b}
v_{K_P^{*,h}}
\int_{-1}^{1}
\left[
F(\gamma_{K,R}^h)
-
F(\gamma_{K,R})
\right]\cdot N_{K,R}
\,{\rm d}\xi
\right.
\nonumber\\
&\qquad\left.
+
\sum_{P\in\mathcal N^b}
v_{K_P^{*,h}}
\int_{-1}^{1}
F(\gamma_{K,R}^h)
\cdot
\left(
N_{K,R}^h-N_{K,R}
\right)
\,{\rm d}\xi
\right|
\nonumber\\
&\le
C
\|v_h\|_{0,\Gamma_h}
\Big(
\sum_{(K,R)}
\int_{-1}^{1}
\left|
F(\gamma_{K,R}^h)
-
F(\gamma_{K,R})
\right|^2
N_{K,R}\,{\rm d}\xi
\Big)^{1/2}
\nonumber\\
&\quad+
C
\big\|
\frac{N_{K,R}^h-N_{K,R}}
{J_{K,R}^h}
\big\|_{\infty}
\|v_h\|_{0,\Gamma_h}
\|F\|_{0,\Gamma_h}
\nonumber\\
&\le
C h^\alpha
\|F\|_{2,U_\Gamma}
\|v_h\|_{0,\Gamma_h}
+
C h^{\alpha-1}
\|F\|_{1,U_\Gamma}
\|v_h\|_{0,\Gamma_h}
\nonumber\\
&\le
C
\left(
h^\alpha\|F\|_{2,U_\Gamma}
+
h^{\alpha-1}\|F\|_{1,U_\Gamma}
\right)
\|v_h\|_{\mathcal T_h^*}.
\label{M1-est}
\end{align}
Similarly, for $M_2$, we can obtain the following estimate 
\begin{align}
\left|
\sum_{P\in\mathcal N^b}
v_{K_P^{*,h}}M_2
\right|\le
C
\left(
h^\alpha\|H\|_{2,U_\Gamma}
+
h^{\alpha-1}\|H\|_{1,U_\Gamma}
\right)
\|v_h\|_{\mathcal T_h^*}.
\label{M2-est}
\end{align}
Combining \eqref{DP-est}, \eqref{M1-est} and
\eqref{M2-est} gives
\begin{align}
|\mathcal G_h^R|
&\le
C h^\alpha
\Big(
\|\nabla\cdot F\|_{1,U_\Gamma}
+
\|F\|_{2,U_\Gamma}
+
\|H\|_{2,U_\Gamma}
\Big)
\|v_h\|_{\mathcal T_h^*}
\nonumber\\
&\quad+
C h^{\alpha-1}
\Big(
\|F\|_{1,U_\Gamma}
+
\|H\|_{1,U_\Gamma}
\Big)
\|v_h\|_{\mathcal T_h^*}\nonumber\\
&\le
Ch^{\alpha-1}
\Big(
\|F\|_{2,U_\Gamma}
+
\|H\|_{2,U_\Gamma}
\Big)
\|v_h\|_{\mathcal T_h^*}
\label{qq2}
\end{align}
for sufficiently small $h$. Substituting \eqref{eqq87} and
\eqref{qq2} into \eqref{qq}, we have
\begin{align}
\left|
a_h(u_h-u_I,v_h)
\right|
&\le
C
\left(
h^k\|u\|_{k+1,\Omega}
+
h^{\alpha-1}\mathcal M_{\Gamma}
\right)
\|v_h\|_{\mathcal T_h^*},
\label{pre-stability}
\end{align}
which together with \eqref{co2} leads to
\begin{align*}
\|u_h-u_I\|_{1,\Omega_h}
\le
C
\sup_{0\ne v_h\in\mathcal V_h}
\frac{
a_h(u_h-u_I,v_h)
}{
\|v_h\|_{\mathcal T_h^*}
}
\leq
C\left(
h^k\|u\|_{k+1,\Omega}
+
h^{\alpha-1}\mathcal M_{\Gamma}
\right).
\end{align*}
Finally, we can prove \eqref{eqq28} by using the triangle inequality and the interpolation estimate.
\end{proof}

%

\subsection{$L^2$ error analysis}
Our \(L^2\) error estimate  is based on the orthogonality condition, which is defined as follows in one dimension; see also \cite{wll21,wz21}.

\begin{definition}\label{defin}
A finite volume scheme  or a piecewise linear mapping $\Pi^{k,*}_{h,x}$  on an interval $I$ is said to satisfy $k$-$(k-1)$-order orthogonal condition if 
\begin{equation}\label{l1}
\int_{I} g(w-\Pi^{k,*}_{h,x}w) {\rm d}x=0,\qquad \forall g\in P_{k-1}(I),~\forall w\in P_1(I),
\end{equation}
where $P_k$ denote the $k$-order polynomial space. 
\end{definition}

\begin{definition}
Assume that the orthogonality condition holds for a finite volume scheme or the mapping $\Pi^{k,*}_h =\Pi^{k,*}_{h,x}\circ\Pi^{k,*}_{h,y}$. Then both component mappings $\Pi^{k,*}_{h,x}$ and $\Pi^{k,*}_{h,y}$ satisfy \eqref{l1}.
\end{definition}

Here, we state the following useful lemma, which is available in \cite{wz21,zz14}.
\begin{lemma}\label{lemma5}
Let $\mathcal T_h$ be the regular curved-edge quadrilateral partition of $\Omega_h$. Assume that the finite volume solution $\widehat u_h$ on
$\widehat K$ satisfies the orthogonality condition. For any $u\in H^{k+2}(\Omega_h)$, $u_h\in\mathcal U_h^k$ and $w\in H^2(\Omega_h)$, there exists a positive constant $C$
independent of $h$, such that
\begin{equation}\label{ll1}
|F_1+F_2|
\le
Ch^{k+1}
\|u\|_{k+2,\Omega_h}
\|w\|_{2,\Omega_h},
\end{equation}
where
\begin{align*}
F_1
&=
\sum_{K^h\in\mathcal T_h}
\int_{\widehat K}
\widehat\nabla\cdot
\left(
\overline{\mathbb D}_K
\widehat\nabla
(\widehat u-\widehat u_h)
\right)
\left(
\Pi_h^{k,*}(\Pi_h^1w)-\Pi_h^1w
\right)
\,{\rm d}\xi\,{\rm d}\eta,
\\
F_2
&=
-\sum_{K^h\in\mathcal T_h}
\int_{\partial\widehat K}
\left(
\overline{\mathbb D}_K
\widehat\nabla
(\widehat u-\widehat u_h)
\right)
\cdot\widehat{\bm n}
\left(
\Pi_h^{k,*}(\Pi_h^1w)-\Pi_h^1w
\right)
\,{\rm d}\widehat s.
\end{align*}
Here $\widehat u=u|_{K^h}\circ\Psi_K$, $\widehat u_h=u_h|_{K^h}\circ\Psi_K$,
$\widehat{\bm n}$ denotes the unit outward normal vector on
$\partial\widehat K$, and  $\overline{\mathbb D}_K$ is a constant matrix.
\end{lemma}

We consider the dual Robin problem: find $z\in H^2(\Omega_h)$ such that
\begin{equation}\label{dual-Robin}
A_h^R(w,z)
=
(\widetilde u-u_h,w)_{\Omega_h},
\qquad
\forall w\in H^1(\Omega_h),
\end{equation}
where
\[
A_h^R(w,z)
=
\int_{\Omega_h}
\kappa_h\nabla w\cdot\nabla z\,{\rm d}x{\rm d}y
+
\int_{\Gamma_h}
\sigma_hwz\,{\rm d}s.
\]
We have the following regularity estimate 
\begin{equation}\label{dual-regularity}
\|z\|_{2,\Omega_h}
\le
C\|\widetilde u-u_h\|_{0,\Omega_h}.
\end{equation}

\begin{theorem}\label{thm:L2-Robin}
Assume that the hypotheses of Theorem \ref{theorem4} hold. Let $u\in H^{k+2}(\Omega)$ and $u_h\in\mathcal U_h^k$ solve 
\eqref{model} and \eqref{scheme} respectively, and suppose $\Gamma$ is  piecewise $C^2$ smooth. Then we have
\begin{equation}\label{L2-Robin-final}
\|\widetilde u-u_h\|_{0,\Omega_h}
\le
C\left(
h^{k+1}\|u\|_{k+2,\Omega}
+
h^\alpha\mathcal M_\Gamma
\right).
\end{equation}
\end{theorem}
\begin{proof}
Denote $e_h=\widetilde u-u_h$. For the solution $z$ of dual Robin problem \eqref{dual-Robin}, Let $\widetilde{z}\in H^2(\widetilde{\Omega})$ be a uniformly bounded
extension of $z$. Define $z_I$ element by element as $z_I=\Pi_h^1\widetilde{z}$, that is,
\[
    z_I|_{K^h}\circ\Psi_K
    =\widehat\Pi_1(\widetilde z\circ\Psi_K),
    \qquad K^h\in\mathcal T_h,
\]
where $\widehat\Pi_1$ is the
bibilinear Lagrangian interpolation operator on the reference element $\widehat K$. By the boundary compatibility
condition, $z_I$ is globally continuous on $\overline{\Omega_h}$. Moreover, 
since $Q_1(\widehat K)\subset Q_k(\widehat K)$, we have
$z_I\in\mathcal U_h^k$. Set $z_h^*=\Pi_h^{k,*}z_I\in\mathcal V_h$. Taking $w=e_h$ in \eqref{dual-Robin} gives
\begin{align}
\|e_h\|_{0,\Omega_h}^2=
A_h^R(e_h,z)=
A_h^R(e_h,z-z_I)
+
A_h^R(e_h,z_I).
\label{L2-main-split}
\end{align}
A direct estimation yields
\begin{align}
|A_h^R(e_h,z-z_I)|
\le
C
\|e_h\|_{1,\Omega_h}
\|z-z_I\|_{1,\Omega_h}
\le
Ch
\|e_h\|_{1,\Omega_h}
\|z\|_{2,\Omega_h},
\label{L2-first-1}
\end{align}
which together with \eqref{eqq28} implies
\begin{align}
|A_h^R(e_h,z-z_I)|
&\le
C\left(
h^{k+1}\|u\|_{k+1,\Omega}
+
h^\alpha\mathcal M_\Gamma
\right)
\|z\|_{2,\Omega_h}.
\label{L2-first-2}
\end{align}
We decompose $A_h^R(e_h,z_I)$ as
\begin{align}
A_h^R(e_h,z_I)=
\mathcal C_h(e_h,z_I)
+
\mathcal G_h^R(\widetilde u,z_h^*),
\label{L2-second-split}
\end{align}
where
\begin{align}
\mathcal C_h(e_h,z_I)
&=
A_h^R(e_h,z_I)
-
a_h(e_h,z_h^*),
\label{Ch-def}
\\
\mathcal G_h^R(\widetilde u,z_h^*)
&=
a_h(\widetilde u,z_h^*)
-
\ell_h(z_h^*).
\label{GR-L2-def}
\end{align}
We first bound $\mathcal C_h(e_h,z_I)$. By the definition of $a_{h,0}$, we have
\begin{align}
\mathcal C_h(e_h,z_I)=&
\int_{\Omega_h}
\kappa_h\nabla e_h\cdot\nabla z_I\,{\rm d}x{\rm d}y
-a_{h,0}(e_h,z_h^*)\nonumber\\
&+
\int_{\Gamma_h}
\sigma_h e_hz_I\,{\rm d}s
-r_h(e_h,z_h^*).\label{ch}
\end{align}
Integration by parts gives
\begin{align}
&\int_{\Omega_h}
\kappa_h\nabla e_h\cdot\nabla z_I\,{\rm d}x{\rm d}y-a_{h,0}(e_h,z^*_h)\nonumber\\
=&\int_{\Omega_h}
\kappa_h\nabla e_h\cdot\nabla z_I\,{\rm d}x{\rm d}y+
\sum_{P\in\mathcal N}
\int_{\partial_0K_P^{*,h}}
\kappa_h(x,y)\frac{\partial e_h}{\partial\bm n}z^*_h
\,{\rm d}s\nonumber\\
=&\sum_{K^h\in\mathcal T_h}\int_{K^h}\nabla\cdot(\kappa_h(x,y)\nabla e_h)(z^*_h-z_I){\rm d}x{\rm d}y\nonumber\\
&\quad-\sum_{K^h\in\mathcal T_h}\int_{\partial K^h}\kappa_h(x,y)\frac{\partial e_h}{\partial\bm{n}}(z^*_h-z_I){\rm d}s\nonumber\\
\nonumber=&\sum_{K^h\in\mathcal T_h}\int_{\widehat{K}}\mathbb{J}^{-1}_K\widehat{\nabla}\cdot\big(\widehat \kappa_K\mathbb{J}^{-1}_K\widehat{\nabla}(\widehat e_h)\big)(z^*_h-z_I)J_K {\rm d}\xi {\rm d}\eta\nonumber\\
\nonumber&\quad -\sum_{K^h\in\mathcal T_h}\int_{\partial \widehat{K}}\widehat \kappa_K\mathbb{J}^{-1}_K\widehat{\bm n}\cdot\big(\mathbb{J}^{-1}_K\widehat{\nabla}(\widehat e_h)\big)(z^*_h-z_I)J_K{\rm d}\widehat{s}\\
\triangleq& E_1+E_2.\label{l3}
\end{align}
We further decompose $E_1$ and $E_2$ as
\[
E_1
=
E_1^1+E_1^2+E_1^3+E_1^4,
\qquad
E_2
=
E_2^1+E_2^2,
\]
where
\begin{align*}
E^1_1
&=
\sum_{K\in\mathcal T_h}
\int_{\widehat{K}}
(\mathbb{J}^{-1}_K-\overline{\mathbb{J}^{-1}_K})
\widehat{\nabla}\cdot
\bigl(
\widehat \kappa_K
\mathbb{J}^{-1}_K
\widehat{\nabla}(\widehat e_h)
\bigr)
(z^*_h-z_I)
J_K
\,{\rm d}\xi{\rm d}\eta,
\\
E^2_1
&=
\sum_{K\in\mathcal T_h}
\int_{\widehat{K}}
\overline{\mathbb{J}^{-1}_K}
\widehat{\nabla}\cdot
\bigl(
\widehat \kappa_K
\mathbb{J}^{-1}_K
\widehat{\nabla}(\widehat e_h)
\bigr)
(z^*_h-z_I)
(J_K-\bar{J}_K)
\,{\rm d}\xi{\rm d}\eta,
\\
E^3_1
&=
\sum_{K\in\mathcal T_h}
\int_{\widehat{K}}
\overline{\mathbb{J}^{-1}_K}
\widehat{\nabla}\cdot
\bigl(
\widehat \kappa_K
(\mathbb{J}^{-1}_K-\overline{\mathbb{J}^{-1}_K})
\widehat{\nabla}(\widehat e_h)
\bigr)
(z^*_h-z_I)
\bar{J}_K
\,{\rm d}\xi{\rm d}\eta,
\\
E^4_1
&=
\sum_{K\in\mathcal T_h}
\int_{\widehat{K}}
\widehat{\nabla}\cdot
\bigl(
\overline{\mathbb{D}}_K
\widehat{\nabla}(\widehat e_h)
\bigr)
(z^*_h-z_I)
\,{\rm d}\xi{\rm d}\eta,
\\
E^1_2
&=
-\sum_{K\in\mathcal T_h}
\int_{\partial\widehat{K}}
\bigl(
(\mathbb{D}_K-\overline{\mathbb{D}}_K)
\widehat{\nabla}(\widehat e_h)
\bigr)
\cdot\widehat{\bm n}
(z^*_h-z_I)
\,{\rm d}\widehat{s},
\\
E^2_2
&=
-\sum_{K\in\mathcal T_h}
\int_{\partial\widehat{K}}
\bigl(
\overline{\mathbb{D}}_K
\widehat{\nabla}(\widehat e_h)
\bigr)
\cdot\widehat{\bm n}
(z^*_h-z_I)
\,{\rm d}\widehat{s},
\end{align*}
with $\mathbb D_K=\mathbb J_K^{-T}\widehat \kappa_K\mathbb J_K^{-1}J_K$ and 
$\overline{\mathbb D}_K=\overline{\mathbb J_K^{-T}\widehat \kappa_K\mathbb J_K^{-1}J_K}$.
Using trace inequalities, inverse estimates, the $H^1$ error bound \eqref{eqq28}, and Lemma~\ref{lem:geometric_coefficient_bounds}, we can obtain 
\begin{align}
|E^1_1|
&\le
Ch^{k+1}
\|u\|_{k+1,\Omega}
\|z\|_{2,\Omega_h},
\label{l4}
\\
|E^2_1|
&\le
Ch^{k+1}
\|u\|_{k+1,\Omega}
\|z\|_{2,\Omega_h},
\label{l5}
\\
|E^3_1|
&\le
Ch^{k+1}
\|u\|_{k+1,\Omega}
\|z\|_{2,\Omega_h},
\label{l6}
\\
|E^1_2|
&\le
Ch^{k+1}
\|u\|_{k+1,\Omega}
\|z\|_{2,\Omega_h}.
\label{l7}
\end{align}
It follows from Lemma~\ref{lemma5} that one has
\begin{equation}
\label{l8}
|E^4_1+E^2_2|
\le
Ch^{k+1}
\|u\|_{k+2,\Omega}
\|z\|_{2,\Omega_h}.
\end{equation}
Combining \eqref{l4}--\eqref{l8} yields
\begin{align}
\left|
\int_{\Omega_h}
\kappa_h\nabla e_h\cdot\nabla z_I\,{\rm d}x{\rm d}y
-
a_{h,0}(e_h,z_h^*)
\right|\le
Ch^{k+1}
\|u\|_{k+2,\Omega}
\|z\|_{2,\Omega_h}.
\label{interp}
\end{align}
For the boundary term, we directly get
\begin{align}
\Big|
\int_{\Gamma_h}
\sigma_he_hz_I\,{\rm d}s
-r_h(e_h,z_h^*)
\Big|&=
\Big|
\int_{\Gamma_h}
\sigma_he_hz_I\,{\rm d}s
-
\sum_{P\in\mathcal N^b}
z_{K_P^{*,h}}^*
\int_{\Gamma_{P,h}}
\sigma_he_h\,{\rm d}s
\Big|
\nonumber\\
&\leq
\int_{\Gamma_h}\left|
\sigma_he_h
(z_I-z_h^*)\right|
\,{\rm d}s
\nonumber\\
&\le
C
\|e_h\|_{0,\Gamma_h}
\|z_I-z_h^*\|_{0,\Gamma_h}\nonumber\\
&\leq Ch
\|e_h\|_{1,\Omega_h}
\|z\|_{2,\Omega_h},
\label{Robin-consistency-1}
\end{align}
which together with \eqref{eqq28}, \eqref{ch} and \eqref{interp} implies
\begin{align}
|\mathcal C_h(e_h,z_I)|
&\le
C
\left(
h^{k+1}\|u\|_{k+2,\Omega}
+
h^\alpha\mathcal M_\Gamma
\right)
\|z\|_{2,\Omega_h}.
\label{consistency-L2-final}
\end{align}

We next turn to the geometric residual $\mathcal G_h^R(\widetilde u,z_h^*)$.
Recall that $F=\widetilde\kappa\nabla\widetilde u$ and $H=\widetilde\sigma\,\widetilde u-\widetilde g$. It follows from the definitions of $a_h$ and $\ell_h$ that
\begin{align}
\mathcal G_h^R(\widetilde u,z_h^*)
={}&
-\sum_{P\in\mathcal N}
z_{K_P^{*,h}}^*
\int_{\partial_0K_P^{*,h}}
F\cdot\bm n_P\,{\rm d}s
-
\sum_{P\in\mathcal N}
z_{K_P^{*,h}}^*
\int_{K_P^{*,h}}
f_h\,{\rm d}x{\rm d}y
\nonumber\\
&+
\sum_{P\in\mathcal N^b}
z_{K_P^{*,h}}^*
\int_{\Gamma_{P,h}}
H\,{\rm d}s,\nonumber\\
={}&
\sum_{P\in\mathcal N^b}
z_{K_P^{*,h}}^*
\Big(
\int_{\Gamma_{P,h}}
(F\cdot\bm n^h+H)\,{\rm d}s
-
\int_{D_P^+}
\nabla\cdot F\,{\rm d}x{\rm d}y
\Big),\label{geo-residual-exact}
\end{align}
where $z_{K_P^{*,h}}^*=z^*_h|_{K_P^{*,h}}$.
By the definition of $S_h^+$, we can obtain
\begin{align}
\mathcal G_h^R(\widetilde u,z_h^*)
={}&
\Big(
\int_{\Gamma_h}
\widetilde z
(F\cdot\bm n^h+H)\,{\rm d}s
-
\int_{S_h^+}
\widetilde z\,\nabla\cdot F\,{\rm d}x{\rm d}y
\Big)
\nonumber\\
&+
\int_{\Gamma_h}
(z_h^*-\widetilde z)
(F\cdot\bm n^h+H)\,{\rm d}s-
\int_{S_h^+}
(z_h^*-\widetilde z)
\nabla\cdot F\,{\rm d}x{\rm d}y
\nonumber\\
\triangleq&
G_1+G_2+G_3.
\label{geo-three-parts}
\end{align}
For $G_1$,  we can write
\begin{align}
G_1
={}&
\Big(
\int_{\Gamma_h}
\widetilde zF\cdot\bm n^h\,{\rm d}s
-
\int_{\Gamma}
\widetilde zF\cdot\bm n\,{\rm d}s
-
\int_{S_h^+}
\widetilde z\,\nabla\cdot F\,{\rm d}x{\rm d}y
\Big)
\nonumber\\
&+
\Big(
\int_{\Gamma_h}
\widetilde zH\,{\rm d}s
-
\int_{\Gamma}
\widetilde zH\,{\rm d}s
\Big)
\nonumber\\
\triangleq&
G_{1,F}+G_{1,H}.
\label{G1-split}
\end{align}
Let $S_h^-=\Omega\setminus\overline{\Omega_h}$ and $S_h=S_h^+\cup S_h^-$. Thus we have (see also \cite{gyc25,lmw10,lsl24})
\begin{equation}\label{strip}
\|w\|_{0,S_h}\leq Ch^{\alpha/2}\|w\|_{1,U_{\Gamma}},\qquad  \forall w\in H^1(U_{\Gamma}).
\end{equation}
Applying the divergence theorem to $\widetilde zF$ on $\Omega_h$ and $\Omega$ respectively gives
\begin{align}
&
\int_{\Gamma_h}
\widetilde zF\cdot\bm n^h\,{\rm d}s
-
\int_{\Gamma}
\widetilde zF\cdot\bm n\,{\rm d}s
\nonumber\\
=&
\int_{S_h^+}
\nabla\cdot(\widetilde zF)\,{\rm d}x{\rm d}y
-
\int_{S_h^-}
\nabla\cdot(\widetilde zF)\,{\rm d}x{\rm d}y\nonumber\\
=&\int_{S_h^+}\widetilde z\,\nabla\cdot F{\rm d}x{\rm d}y+\int_{S_h^+}F\cdot\nabla\widetilde z{\rm d}x{\rm d}y-\int_{S_h^-}
\nabla\cdot(\widetilde zF)\,{\rm d}x{\rm d}y.\label{div}
\end{align}
It follows from \eqref{strip-estimate}, \eqref{strip} and \eqref{div} that one has
\begin{align}
|G_{1,F}|
&=
\Big|\int_{S_h^+}
F\cdot\nabla\widetilde z\,{\rm d}x{\rm d}y-
\int_{S_h^-}
\left(
\widetilde z\,\nabla\cdot F
+
F\cdot\nabla\widetilde z
\right)
\,{\rm d}x{\rm d}y\Big|\nonumber\\
&\le
\|F\|_{0,S_h^+}
\|\nabla\widetilde z\|_{0,S_h^+}
+
\|\widetilde z\|_{0,S_h^-}
\|\nabla\cdot F\|_{0,S_h^-}+\|F\|_{0,S_h^-}
\|\nabla\widetilde z\|_{0,S_h^-}
\nonumber\\
&\le
Ch^\alpha
\left(
\|F\|_{1,U_\Gamma}
+
\|\nabla\cdot F\|_{1,U_\Gamma}
\right)
\|\widetilde z\|_{2,\widetilde\Omega}
\nonumber\\
&\le
Ch^\alpha
\|F\|_{2,U_\Gamma}
\|z\|_{2,\Omega_h}.
\label{G1F-identity}
\end{align}
To estimate $G_{1,H}$, on each smooth segment of $\Gamma$, let $\bm\nu: U_\Gamma\rightarrow \mathbb R^2$ such that $|\bm\nu| = 1$ and $\bm\nu = \bm n$ on $\Gamma$.
We decompose $G_{1,H}$ as
\begin{align}
G_{1,H}
={}&\int_{\Gamma_h}
\widetilde zH
\bm\nu\cdot\bm n^h\,{\rm d}s
-
\int_{\Gamma}
\widetilde zH
\bm\nu\cdot\bm n\,{\rm d}s+
\int_{\Gamma_h}
\widetilde zH
\left(
1-\bm\nu\cdot\bm n^h
\right)
\,{\rm d}s.
\label{G1H-split}
\end{align}
Applying the divergence theorem to  $\widetilde zH\bm\nu$ yields
\begin{align}
&
\int_{\Gamma_h}
\widetilde zH
\bm\nu\cdot\bm n^h\,{\rm d}s
-
\int_{\Gamma}
\widetilde zH
\bm\nu\cdot\bm n\,{\rm d}s
\nonumber\\
=&
\int_{S_h^+}
\nabla\cdot(\widetilde zH\bm\nu)\,{\rm d}x{\rm d}y
-
\int_{S_h^-}
\nabla\cdot(\widetilde zH\bm\nu)\,{\rm d}x{\rm d}y.
\end{align}
Since
\[
\nabla\cdot(\widetilde zH\bm\nu)
=
H\bm\nu\cdot\nabla\widetilde z
+
\widetilde z\bm\nu\cdot\nabla H
+
\widetilde zH\nabla\cdot\bm\nu,
\]
the narrow-strip estimate \eqref{strip} implies
\begin{align}
\Big|
\int_{S_h^+}
\nabla\cdot(\widetilde zH\bm\nu)\,{\rm d}x{\rm d}y
-
\int_{S_h^-}
\nabla\cdot(\widetilde zH\bm\nu)\,{\rm d}x{\rm d}y
\Big|\le
Ch^\alpha
\|H\|_{2,U_\Gamma}
\|z\|_{2,\Omega_h}.
\label{G1H-volume}
\end{align}
In fact, we have
\[
1-\bm\nu\cdot\bm n^h
=
\frac12
|\bm\nu-\bm n^h|^2,
\]
which leads to
\[
\|1-\bm\nu\cdot\bm n^h\|_{\infty,\Gamma_h}
\le
Ch^{2\alpha-2}.
\]
Therefore, by the trace theorem, we get
\begin{align}
\left|
\int_{\Gamma_h}
\widetilde zH
(1-\bm\nu\cdot\bm n^h)\,{\rm d}s
\right|
&\le
Ch^{2\alpha-2}
\|H\|_{1,U_\Gamma}
\|z\|_{2,\Omega_h}
\nonumber\\
&\le
Ch^\alpha
\|H\|_{1,U_\Gamma}
\|z\|_{2,\Omega_h},\label{G1H_b}
\end{align}
where we have used the fact that $\alpha\geq 2$.
Combining \eqref{G1H-volume} with \eqref{G1H_b} yields
\begin{equation}\label{G1H-est}
|G_{1,H}|
\le
Ch^\alpha
\|H\|_{2,U_\Gamma}
\|z\|_{2,\Omega_h},
\end{equation}
which together with \eqref{G1F-identity} yields
\begin{equation}\label{G1-final}
|G_1|
\le
Ch^\alpha(\|F\|_{2,U_\Gamma}
+\|H\|_{2,U_\Gamma})
\|z\|_{2,\Omega_h}.
\end{equation}

We now estimate $G_2$. Using the Cauchy-Schwarz inequality gives
\begin{equation}\label{G2}
|G_2|\le
\|z_h^*-\widetilde z\|_{0,\Gamma_h}
\|Q_h\|_{0,\Gamma_h},
\end{equation}
where $Q_h = F \cdot \bm{n}^h + H$.  For any $x_h \in \Gamma_h$,  we have
\begin{align}
Q_h(x_h)
={}&
\bigl(
F(x_h)-F(x)
\bigr)
\cdot\bm n^h(x_h)
\nonumber\\
&+
F(x)\cdot
\bigl[
\bm n^h(x_h)-\bm n(x)
\bigr]
+
H(x_h)-H(x),
\label{Qh-decomp}
\end{align}
where $x = \Psi^{-1}_K(x_h)\in \Gamma$.
It follows from Assumption~\ref{ass_geo} that 
\[
\|F(x_h)-F(x)\|_{0,\Gamma_h}
+
\|H(x_h)-H(x)\|_{0,\Gamma_h}
\le
Ch^\alpha
\mathcal M_\Gamma,
\]
which leads to
\begin{align}
\|Q_h\|_{0,\Gamma_h}
&\le
Ch^\alpha\mathcal M_\Gamma
+
Ch^{\alpha-1}\|F\|_{1,U_\Gamma}\le
Ch^{\alpha-1}
\mathcal M_\Gamma.
\label{Qh-est}
\end{align}
Furthermore,
\begin{align}
\|z_h^*-\widetilde z\|_{0,\Gamma_h}
\le
\|z_h^*-z_I\|_{0,\Gamma_h}
+
\|z_I-\widetilde z\|_{0,\Gamma_h}\le
Ch
\|z\|_{2,\Omega_h}.
\label{zh-z-est}
\end{align}
Thus, combining \eqref{G2}, \eqref{Qh-est} and \eqref{zh-z-est} yileds
\begin{align}
|G_2|
&\le
\|z_h^*-\widetilde z\|_{0,\Gamma_h}
\|Q_h\|_{0,\Gamma_h}\le
Ch^\alpha
\mathcal M_\Gamma
\|z\|_{2,\Omega_h}.
\label{G2-est}
\end{align}

Finally, we bound $G_3$. It follows from the Cauchy-Schwarz inequality that we have
\begin{equation}\label{G3}
|G_3|\le
\|z_h^*-\widetilde z\|_{0,S_h^+}
\|\nabla\cdot F\|_{0,S_h^+}.
\end{equation}
Since $z_h^*$ is piecewise constant on each control volume $K_P^{*,h}$ and 
$|D_P^+|\le Ch^\alpha|\Gamma_{P,h}|$, we have
\begin{align}
\|z_h^*\|_{0,S_h^+}^2=
\sum_{P\in\mathcal N^b}
|z_{K_P^{*,h}}^*|^2
|D_P^+|\le
Ch^\alpha
\|z\|_{2,\Omega_h}^2.
\label{zh-strip-1}
\end{align}
Combining \eqref{strip-estimate} and \eqref{zh-strip-1} gives
\begin{align}
\|z_h^*-\widetilde z\|_{0,S_h^+}\le
\|z_h^*\|_{0,S_h^+}
+
\|\widetilde z\|_{0,S_h^+}\le
Ch^{\alpha/2}
\|z\|_{2,\Omega_h}.
\label{zh-z-strip}
\end{align}
The narrow-strip estimate \eqref{strip-estimate} further implies
\[
\|\nabla\cdot F\|_{0,S_h^+}
\le
Ch^{\alpha/2}
\|\nabla\cdot F\|_{1,U_\Gamma}
\le
Ch^{\alpha/2}
\|F\|_{2,U_\Gamma},
\]
which together with \eqref{zh-z-strip} yields
\begin{align}
|G_3|\le
\|z_h^*-\widetilde z\|_{0,S_h^+}
\|\nabla\cdot F\|_{0,S_h^+}\le
Ch^\alpha
\mathcal M_\Gamma
\|z\|_{2,\Omega_h}.
\label{G3-est}
\end{align}
Combining \eqref{G1-final}, \eqref{G2-est}, and \eqref{G3-est}, we have 
\begin{equation}\label{geo-L2-final}
\left|
\mathcal G_h^R(\widetilde u,z_h^*)
\right|
\le
Ch^\alpha
\mathcal M_\Gamma
\|z\|_{2,\Omega_h}.
\end{equation}
Finally, it follows from \eqref{L2-main-split}, \eqref{L2-first-2}, \eqref{L2-second-split}, \eqref{consistency-L2-final}, and \eqref{geo-L2-final} that we obtain
\begin{align}
\|e_h\|_{0,\Omega_h}^2
&\le
C
\left(
h^{k+1}\|u\|_{k+2,\Omega}
+
h^\alpha\mathcal M_\Gamma
\right)
\|z\|_{2,\Omega_h}
\nonumber\\
&\le
C
\left(
h^{k+1}\|u\|_{k+2,\Omega}
+
h^\alpha\mathcal M_\Gamma
\right)
\|e_h\|_{0,\Omega_h}.
\label{L2-final-step}
\end{align}
Thus we complete the proof.
\end{proof}

\begin{remark}
This paper focuses solely on the Robin boundary value problem; research regarding the Dirichlet and Neumann boundary value problems will be addressed in future work.
\end{remark}

\section{Numerical results}

To validate the theoretical results, we present three sets of numerical experiments on curved-edge quadrilateral meshes. The first set of experiments demonstrates that, regardless of whether the diffusion coefficient is constant or anisotropic matrix, the proposed method achieves optimal convergence order on various curved-edge meshes. The second set of experiments demonstrates that constructing dual meshes based on Gaussian points is crucial for achieving high accuracy and optimal convergence, whereas dual meshes constructed using equidistant points lead to a significant decline in accuracy. The third set of experiments examines two severely deformed domains and an interface problem, with the results further validating the significant advantages of curved-edge meshes in approximating curved-edge domains and capturing highly curved interfaces.

\begin{example}
We consider the second-order elliptic
Robin boundary value problem \eqref{model} on the domain
$\Omega=[-1,1]^2$. The exact solution and Robin coefficients are
chosen as 
\begin{equation*}
u(x,y)=2+\sin(\pi x)\sin(\pi y),\
\sigma=2, \ \kappa=1 \ {\rm or} \ \kappa=A=
\begin{pmatrix}
10 & 2\\
2 & 1
\end{pmatrix}.
\end{equation*}
The source term
$f$ and the Robin boundary data $g$ are determined by the exact
solution. We employ two distinct parametric mappings  
\begin{align}
&\Psi_1=\left\{
\begin{aligned}
&x(\xi,\eta)=\xi+0.5\eta(1-\xi^2)^2(1-\eta^2),\\
&y(\xi,\eta)=\eta-0.5\xi(1-\xi^2)(1-\eta^2)^2,
\end{aligned}
\right. \label{map1}\\
&\Psi_2=\left\{
\begin{aligned}
&x(\xi,\eta)=\xi+0.1\cos(\frac{\pi\xi}{2})\cos(\frac{3\pi \eta}{2}),\\
&y(\xi,\eta)=\eta+0.1\sin(2\pi\xi)\cos(\frac{\pi\eta}{2}),
\end{aligned}
\right. \label{map2}
\end{align}
to generate the curved-edge meshes; see also \cite{ca22,ps08}. These meshes are showed in Figure \ref{fig:a1}. 

For each mesh, we compute the errors \( \|u-u_h\|_{1,\Omega_h} \) and \( \|u-u_h\|_{0,\Omega_h} \) under different choices of the diffusion coefficient. From Figures~\ref{fig2}--\ref{fig3}, we observe that the \(H^1\)-error \( \|u-u_h\|_{1,\Omega_h} \) exhibits the expected convergence rate \(\mathcal O(h^k)\), where \(k\) is the polynomial degree of the trial space. Similarly, the \(L^2\)-error \( \|u-u_h\|_{0,\Omega_h} \) converges with order \(k+1\). These results confirm that the proposed scheme remains stable and accurate for both constant and anisotropic coefficients. It is worth noting that this two mappings \(\Psi_1\) and \(\Psi_2\) used in this test do not necessarily satisfy the strengthened asymptotic affine assumption imposed in the theoretical analysis.  Nevertheless, the numerical results still show optimal convergence. This suggests that the imposed mapping assumption is a sufficient condition for the proof, but it is not necessary in practice.

\begin{figure}[htbp]  
    \centering  
    \begin{subfigure}{0.48\textwidth}
        \centering
        \includegraphics[width=\textwidth]{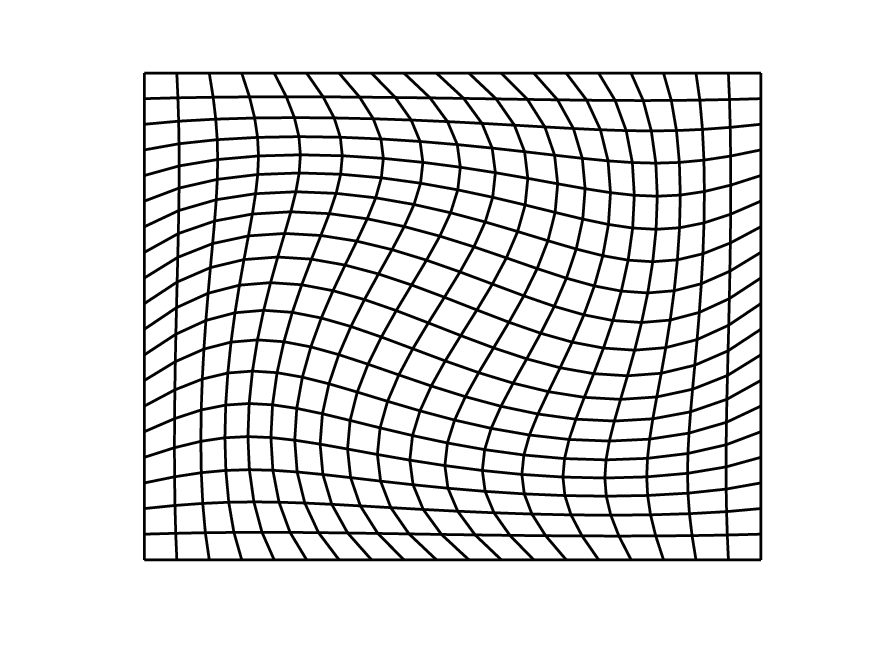}
        \caption{Mesh generated by $\Psi_1$}
    \end{subfigure}
    \hfill
    \begin{subfigure}{0.48\textwidth}
        \centering
        \includegraphics[width=\textwidth]{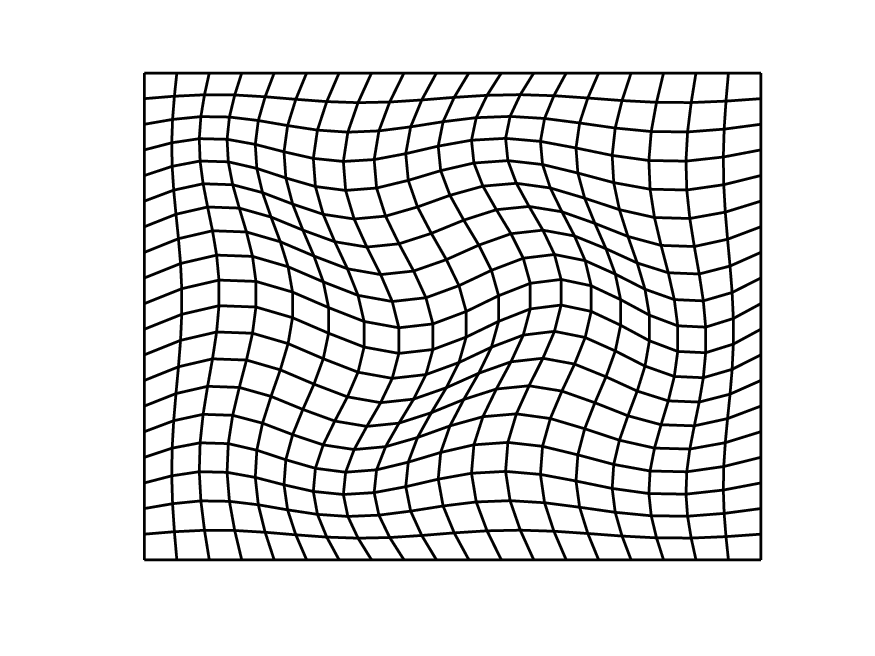}
        \caption{Mesh generated by $\Psi_2$}
    \end{subfigure}
    
    \caption{Curved-edge meshes generated by mapping  $\Psi_1$ and $\Psi_2$ }
    \label{fig:a1}
\end{figure}

\begin{figure}[htbp]  
    \centering  
    \begin{subfigure}{0.48\textwidth}
        \centering
        \includegraphics[width=\textwidth]{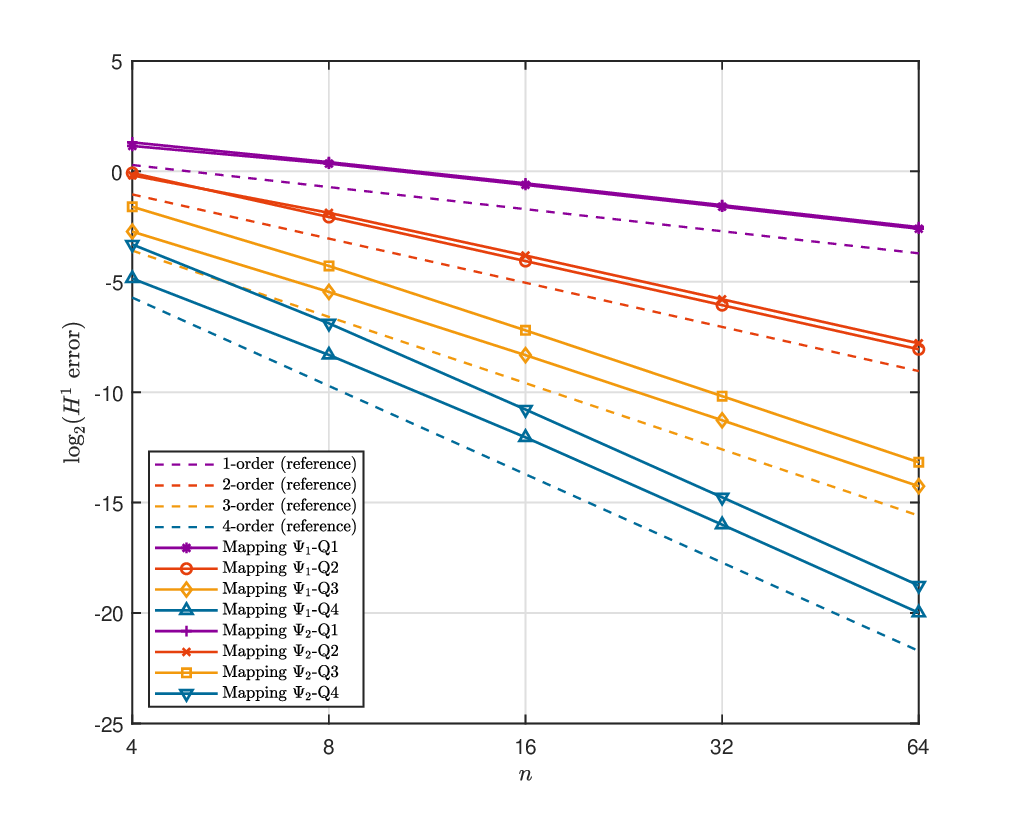}
    \end{subfigure}
    \hfill
    \begin{subfigure}{0.48\textwidth}
        \centering
        \includegraphics[width=\textwidth]{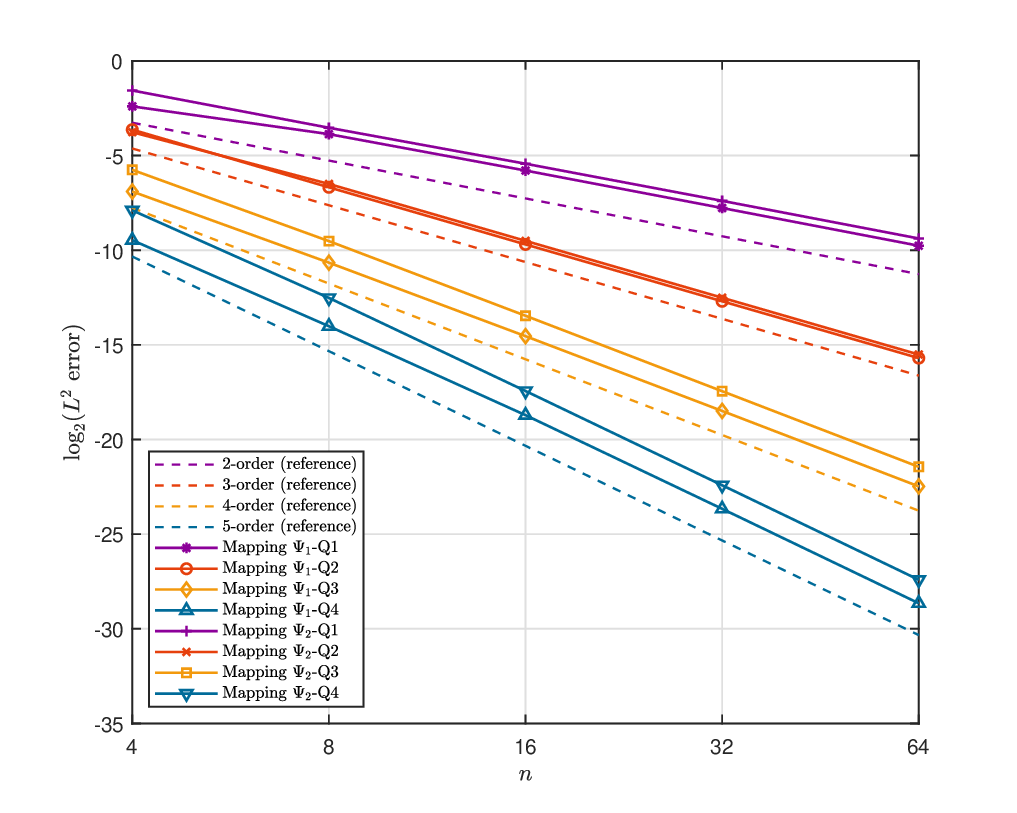}
    \end{subfigure}
    
    \caption{$H^1$ and $L^2$ errors in $\mathbb{Q}_1$-$\mathbb{Q}_4$ spaces for mapping $\Psi_1$-$\Psi_2$ when $\kappa=1$ }
    \label{fig2}
\end{figure}

\begin{figure}[htbp]  
    \centering  
    \begin{subfigure}{0.48\textwidth}
        \centering
        \includegraphics[width=\textwidth]{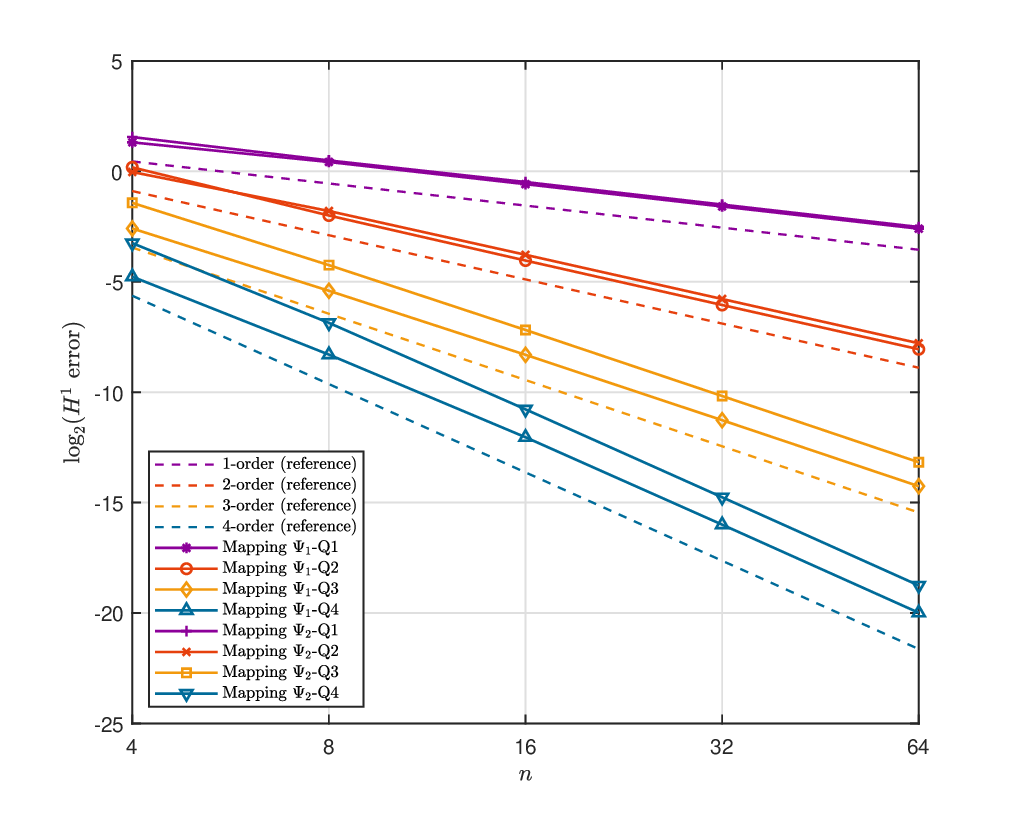}
    \end{subfigure}
    \hfill
    \begin{subfigure}{0.48\textwidth}
        \centering
        \includegraphics[width=\textwidth]{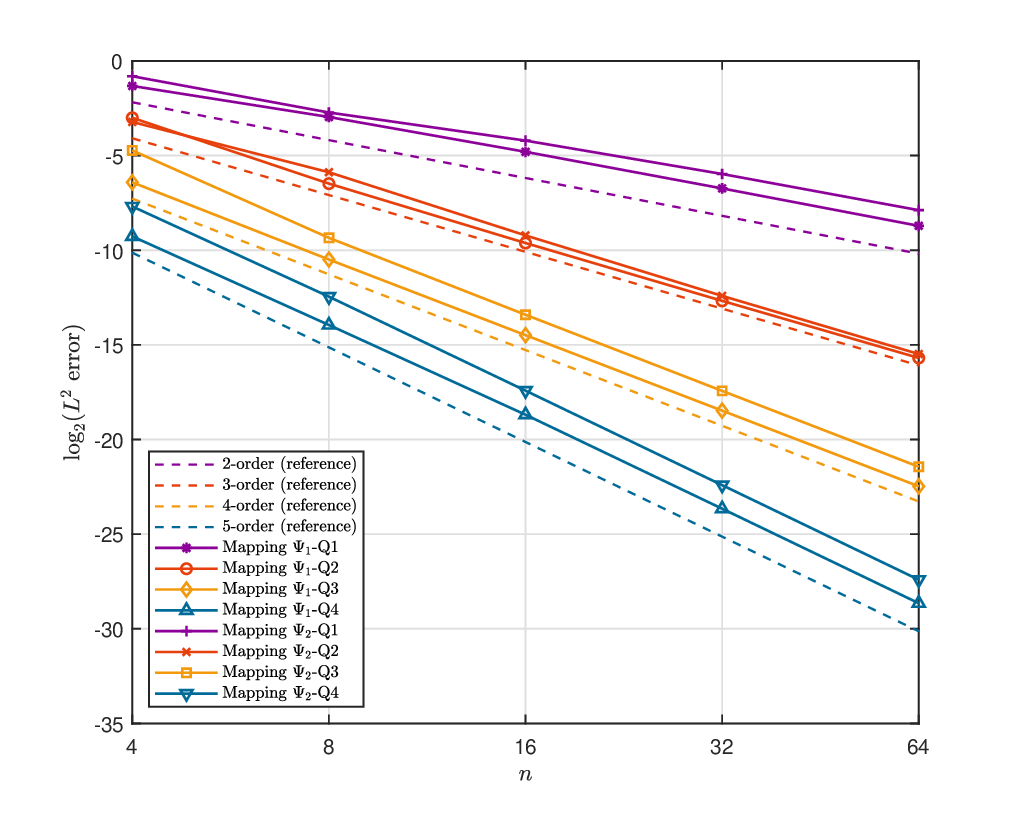}
    \end{subfigure}
    
    \caption{$H^1$ and $L^2$ errors in $\mathbb{Q}_1$-$\mathbb{Q}_4$ spaces for mapping $\Psi_1$-$\Psi_2$ when $\kappa=A$ }
    \label{fig3}
\end{figure}

\end{example}

\begin{example} 
This example uses the same solution and coefficients as Example 1. In dual partitioning, the most commonly used strategies for selecting dual nodes are based on equidistant points and Gaussian points (as shown in Figure \ref{fig:dual}). However, not all dual partitions achieve optimal convergence. Therefore, we selected these two typical schemes to conduct a comparative analysis of their numerical performance. To this end, we employ a mapping function taken from \cite{cms25}, defined by
\begin{equation}
\Psi_3=\left\{
\begin{aligned}
&x(\xi,\eta)=\xi+0.05\sin(4\pi\xi)\sin(4\pi\eta),\\
&y(\xi,\eta)=\eta+0.05\sin(4\pi\xi)\sin(4\pi\eta).
\end{aligned}
\right.
\end{equation}

\begin{figure}
\centering
\includegraphics[width=0.8\textwidth]{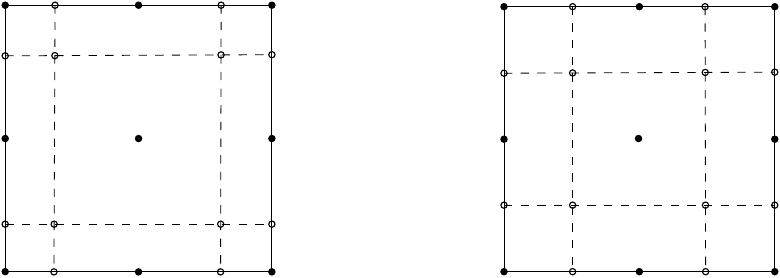}
\caption{Different dual partitions on $\widehat{K}$, 'dot' are interpolation points and 'circle' are dual points. Left: Gaussian dual partition. Right: equidistant dual partition }
\label{fig:dual}
\end{figure}

As shown in Figure~\ref{fig:a5}, the choice of dual partition points has a significant influence on the numerical accuracy. The solution obtained with the equidistant dual partition is visibly less accurate than that obtained with the Gaussian dual partition. This difference is also reflected in Table~\ref{exm33}, when equidistant points are used, the expected optimal rate \(\mathcal O(h^3)\) is not achieved. In contrast, the Gaussian dual partition yields the desired convergence behavior. Therefore, in all subsequent numerical tests, we construct the dual mesh using Gaussian quadrature points in order to obtain the best observed accuracy and the optimal convergence rate.

\begin{table}[!tbh]
\caption{$L^2$ convergence orders for $\mathbb{Q}_2$ elements with different dual partitions when $\kappa=A$}
\vspace{2mm}
\setlength{\tabcolsep}{4pt} 
\centering
\label{exm33}
\begin{tabular}{ccc|cc}
\hline
 & \multicolumn{2}{c|}{Equidistant partition} & \multicolumn{2}{c}{Gaussian partition} \\\hline
$1/h$ & $\|u-u_h\|_{0,\Omega_h}$ & order & $\|u-u_h\|_{0,\Omega_h}$ & order \\
\hline
$2$ & $3.11e-01$ & $-$ & $2.10e-01$ & $-$ \\
$4$ & $9.06e-02$ & $1.71$ & $5.35e-02$ & $2.53$ \\
$8$ & $1.15e-02$ & $2.96$ & $8.28e-03$ & $2.70$ \\
$16$ & $3.69e-03$ & $1.64$ & $8.94e-04$ & $3.21$ \\
$32$ & $1.00e-03$ & $1.86$ & $9.80e-05$ & $3.18$ \\
\hline
\end{tabular}
\end{table}

\begin{figure}[htbp]  
    \centering  
    \begin{subfigure}{0.32\textwidth}
        \centering
        \includegraphics[width=\textwidth]{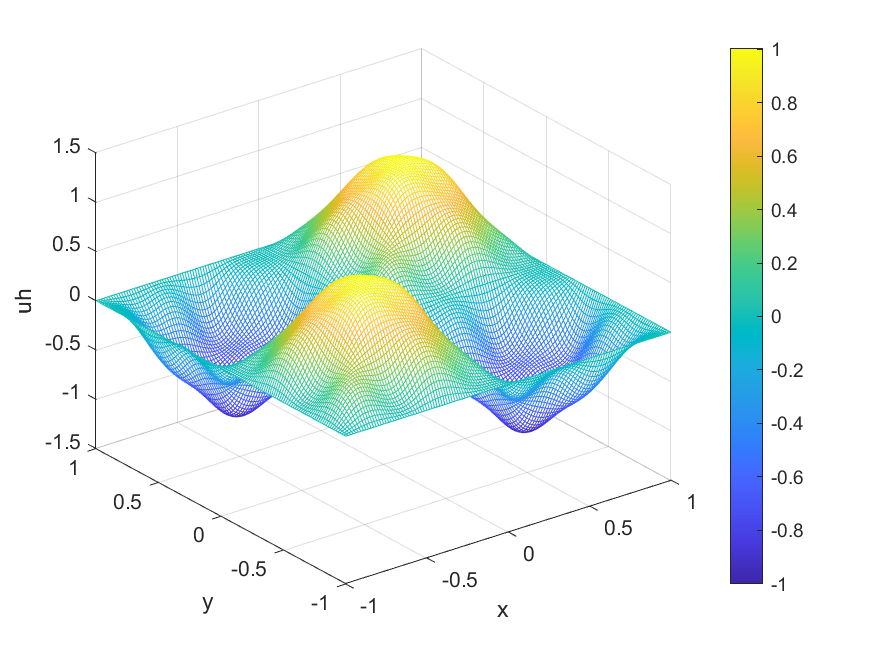}
    \end{subfigure}
    \begin{subfigure}{0.32\textwidth}
        \centering
        \includegraphics[width=\textwidth]{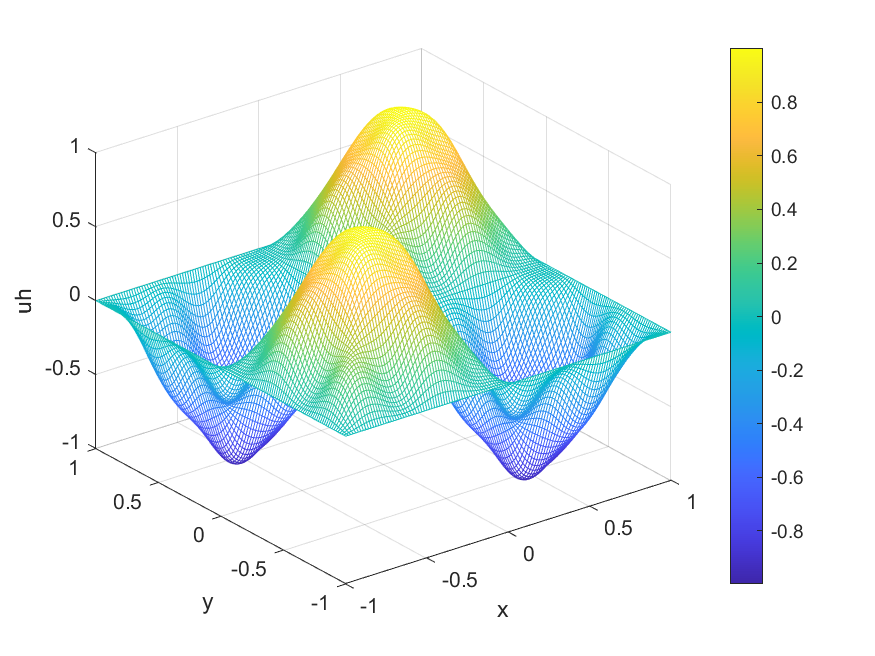}
    \end{subfigure}
        \begin{subfigure}{0.32\textwidth}
        \centering
        \includegraphics[width=\textwidth]{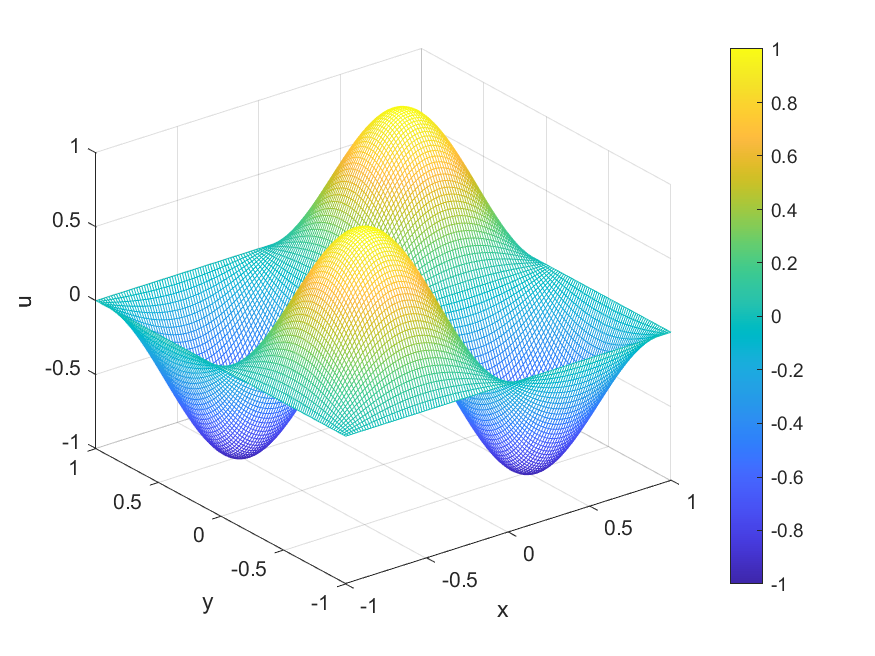}
    \end{subfigure}
    
    \caption{Images of numerical and exact solutions. Left: numerical solution in equidistant dual partition. Middle: numerical solution in Gaussian dual partition. Right: exact solution. }
    \label{fig:a5}
\end{figure}
\end{example}


\begin{example}
In this example, we investigate interpolation approximations for two domains with complex curved boundaries. For simplicity, we consider problems with zero boundary conditions, where $\kappa=1$. If the curved boundary admits an explicit analytical representation, the mapping \(\Psi_K\) can be constructed directly. If the boundary is prescribed only by discrete control nodes, the mapping has to be generated by polynomial interpolation. In this case, the accuracy of the geometric approximation becomes crucial. To retain the optimal convergence rate, the order of the boundary interpolation must be compatible with the smoothness and geometric complexity of the computational domain, otherwise, the geometric error may dominate the discretization error.

We first consider the quadratic annular domain shown in Figure~\ref{fig:a6}(b), which is discretized by a mesh with uniform spacing in both the radial and circumferential directions. The results in Table~\ref{exm5a} show that linear interpolation is not sufficient to approximate the quadratic boundary accurately. As a consequence, the optimal convergence rate is not achieved. In this case, the total error is dominated by the geometric contribution, of order \(h^{\alpha-1}\) in the \(H^1\)-norm and \(h^\alpha\) in the \(L^2\)-norm, rather than by the expected discretization errors \(h^k\) and \(h^{k+1}\). By contrast, quadratic interpolation is able to represent the quadratic boundary with the required accuracy, so that the geometric error no longer dominates and the optimal convergence rates are recovered. A similar phenomenon can be observed in Table~\ref{exm4a} for the domain shown in Figure~\ref{fig:a6}(a). This domain has a highly curved and geometrically complex boundary. Bilinear and bi-quadratic boundary interpolations do not provide sufficient geometric accuracy and therefore fail to deliver the optimal convergence orders. The resulting errors are again dominated by the boundary approximation error. In contrast, bi-cubic interpolation captures the boundary geometry more accurately and restores the optimal convergence behavior, with rates \(\mathcal O(h^4)\) in the \(L^2\)-norm and \(\mathcal O(h^3)\) in the \(H^1\)-norm. This indicates that, for this domain, at least third-order geometric interpolation is needed to resolve the boundary sufficiently well and to recover the optimal accuracy of the numerical scheme.

\begin{figure}[htbp]
\centering
\begin{subfigure}[b]{0.33\textwidth}
    \includegraphics[height=4.5cm, width=\textwidth]{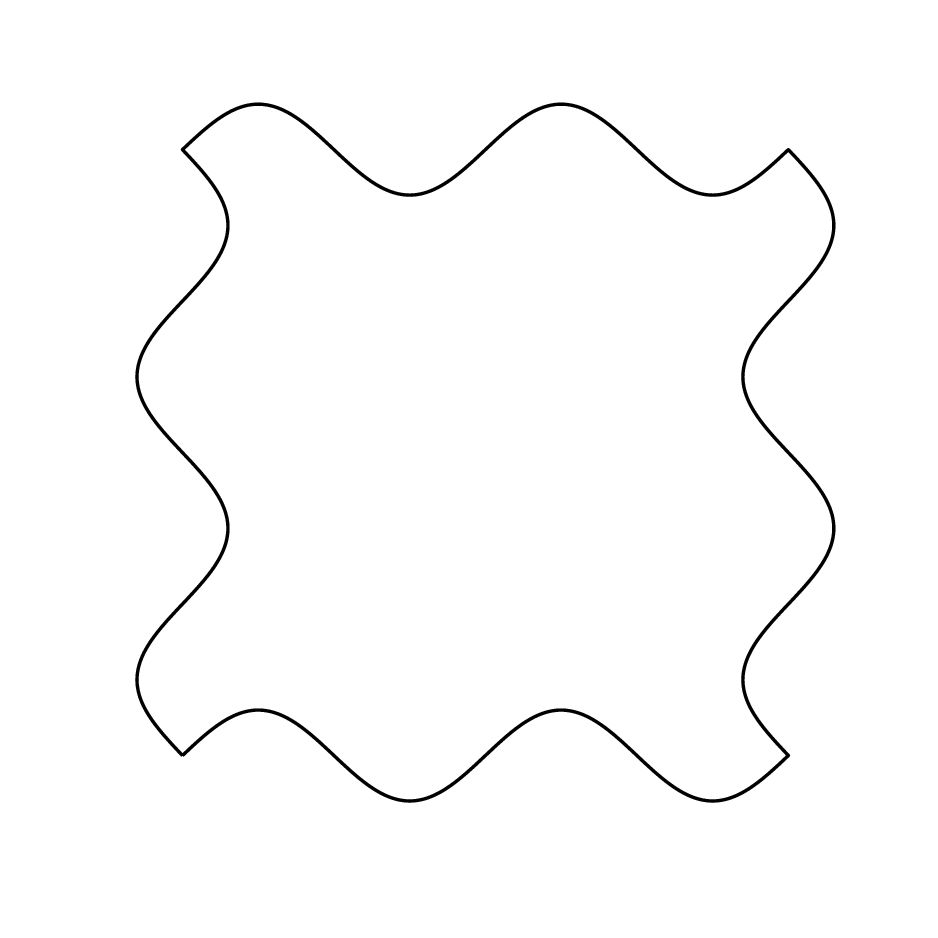}
    \caption{domain 1}
\end{subfigure}
\begin{subfigure}[b]{0.33\textwidth}
    \includegraphics[height=4.5cm, width=\textwidth]{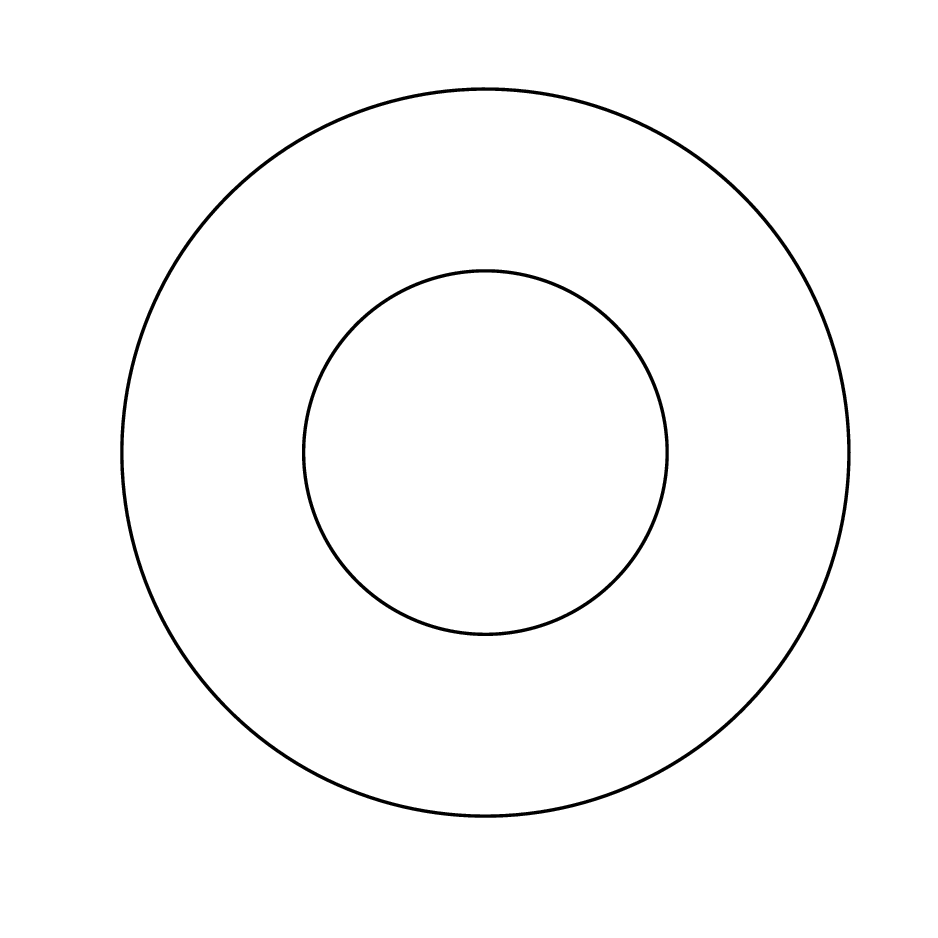}
    \caption{domain 2}
\end{subfigure}
\hfill
\begin{subfigure}[b]{0.30\textwidth}
    \includegraphics[height=4cm, width=\textwidth]{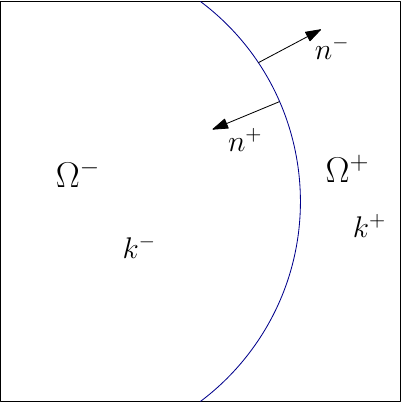}
    \caption{domain 3}
\end{subfigure}
\caption{(a): curved-edge domain, (b): annular domain, (c): interface domain }
\label{fig:a6}
\end{figure}

\begin{table}[!tbh]
\caption{ Numerical results of $\mathbb{Q}_2$ element for annular domain in Figure \ref{fig:a6}(b) using different interpolation }
\vspace{2mm}
\centering
\label{exm5a}
{\begin{tabular}{ cccccc}\hline
&$1/h$ & $ \| u-u_h\|_{0,\Omega_h}$ & order & $\| u- u_h\|_{1,\Omega_h}$  & order   \\ \hline
\multirow{5}{*}{Bilinear}  &$2^2\times 2^3$&  $2.4036e+00$ & $-$    & $6.5301e+00$ & $-$\\
                           &$2^3\times 2^4$&  $7.9121e-01$ & $1.60$ & $2.5874e+00$ & $1.34$\\
                           &$2^4\times 2^5$&  $2.0737e-01$ & $1.93$ & $8.8336e-01$ & $1.55$\\
                           &$2^5\times 2^6$&  $5.1874e-02$ & $1.99$ & $2.9873e-01$ & $1.56$\\
                           &$2^6\times 2^7$&  $1.2897e-02$ & $2.01$ & $1.0259e-01$ & $1.54$\\
                           &$2^7\times 2^8$&  $3.2106e-03$ & $2.00$ & $3.5688e-02$ & $1.53$\\
\hline
\multirow{5}{*}{Biquadratic}&$2^2\times 2^3$& $1.0793e-01$ & $-$    & $1.1226e+00$ & $-$\\
                           &$2^3\times 2^4$&  $1.3464e-02$ & $3.11$ & $2.8225e-01$ & $1.99$\\
                           &$2^4\times 2^5$&  $1.6771e-03$ & $3.00$ & $7.0535e-02$ & $2.00$\\
                           &$2^5\times 2^6$&  $2.0934e-04$ & $3.00$ & $1.7630e-02$ & $2.00$\\
                           &$2^6\times 2^7$&  $2.6157e-05$ & $3.01$ & $4.4071e-03$ & $2.00$\\
                           &$2^7\times 2^8$&  $3.2692e-06$ & $3.00$ & $1.1018e-03$ & $2.00$\\
\hline
\end{tabular}}
\end{table}

\begin{table}[!tbh]
\caption{Numerical results of $\mathbb{Q}_3$ element for curved-edge domain in Figure \ref{fig:a6}(a) using different interpolation}
\vspace{2mm}
\centering
\label{exm4a}
{\begin{tabular}{ cccccc}\hline
&$1/h$ & $ \| u-u_h\|_{0,\Omega_h}$ & order & $\| u- u_h\|_{1,\Omega_h}$  & order   \\ \hline
\multirow{5}{*}{Bilinear}  &$2^3$&  $2.3955e-02$ & $2.13$ & $2.3748e-01$ & $2.06$\\
                           &$2^4$&  $5.8982e-03$ & $2.02$ & $6.2332e-02$ & $1.92$\\
                           &$2^5$&  $1.5033e-03$ & $1.98$ & $1.7626e-02$ & $1.82$\\
                           &$2^6$&  $3.7868e-04$ & $1.99$ & $5.3003e-03$ & $1.73$\\
                           &$2^7$&  $9.4964e-05$ & $2.00$ & $1.6903e-03$ & $1.64$\\
\hline
\multirow{5}{*}{Biquadratic}&$2^3$&  $1.9164e-03$ & $2.61$ & $6.6058e-02$ & $2.47$\\
                           &$2^4$&  $1.7597e-04$ & $3.45$ & $1.0444e-02$ & $2.66$\\
                           &$2^5$&  $1.4156e-05$ & $3.64$ & $1.5579e-03$ & $2.74$\\
                           &$2^6$&  $1.1675e-06$ & $3.60$ & $2.3943e-04$ & $2.70$\\
                           &$2^7$&  $9.9148e-08$ & $3.57$ & $3.8630e-05$ & $2.64$\\
\hline
\multirow{5}{*}{Bicubic} &$2^3$&    $1.3496e-03$ & $3.46$ & $7.9747e-02$ & $2.21$\\
                           &$2^4$&  $1.2646e-04$ & $3.42$ & $1.2737e-02$ & $2.65$\\
                           &$2^5$&  $8.3210e-06$ & $3.93$ & $1.6748e-03$ & $2.93$\\
                           &$2^6$&  $5.2825e-07$ & $3.98$ & $2.1232e-04$ & $2.98$\\
                           &$2^7$&  $3.3149e-08$ & $3.99$ & $2.6642e-05$ & $2.99$\\
\hline
\end{tabular}}
\end{table}

Finally, we consider the interface problem shown in Figure~\ref{fig:a6}(c). The interface is described by
\[
\Theta(x,y)=x-0.5\cos\left(\frac{\pi y}{2}\right)=0.
\]
The diffusion coefficient \(\kappa\) is piecewise constant, taking values \(k^-\) and \(k^+\) in \(\Omega^-\) and \(\Omega^+\), respectively. The exact solution is chosen as
\[
u(x,y)=
\begin{cases}
(1-x^2)(1-y^2)\left(x-0.5\cos\left(\frac{\pi y}{2}\right)\right), & x\in\Omega^-,\\[2mm]
\dfrac{k^-}{k^+}(1-x^2)(1-y^2)\left(x-0.5\cos\left(\frac{\pi y}{2}\right)\right), & x\in\Omega^+,
\end{cases}
\]
so that the interface continuity conditions are satisfied. In the computation, we take \(k^+=1000\) and \(k^-=1\), which represents a strong jump in the material coefficient across the interface.

This example illustrates the importance of accurately resolving curved interfaces. When a straight-edge mesh is used, the curved interface \(\Theta=0\) cannot be represented exactly, and the resulting geometric error is amplified by the large coefficient contrast. As shown in Table~\ref{exm6a}, straight edge mesh leads to a loss of accuracy and prevents the method from achieving the optimal convergence order. In contrast, the curved-edge mesh fits the interface much more accurately and significantly reduces the geometric error. The numerical results show that the proposed method recovers the optimal convergence rates in both the \(L^2\)-norm and the \(H^1\)-norm, demonstrating the advantage of curved-edge meshes for interface problems with large coefficient jumps.

\begin{table}[!tbh]
\caption{Numerical results of $\mathbb{Q}_2$ element for interface problems on straight/curved edge meshes }
\vspace{2mm}
\setlength{\tabcolsep}{2pt} 
\centering
\label{exm6a}
\begin{tabular}{ccccc|cccc}
\hline
 & \multicolumn{4}{c|}{Curved-edge mesh} & \multicolumn{4}{c}{Straight-edge mesh} \\
\hline
$1/h$ & $\| u-u_h\|_{0,\Omega_h}$ & order & $\|u-u_h\|_{1,\Omega_h}$ & order & $\| u-u_h\|_{0,\Omega_h}$ & order & $\|u-u_h\|_{1,\Omega_h}$ & order\\
\hline
$2^3$ & $2.78e-01$ & $2.99$ & $5.69e+00$ & $2.02$ & $3.24e-01$ & $2.36$ & $4.34e+00$ & $1.92$\\
$2^4$ & $3.49e-02$ & $2.99$ & $1.42e+00$ & $2.00$ & $7.21e-02$ & $2.19$ & $1.14e+00$ & $1.93$ \\
$2^5$ & $4.37e-03$ & $3.00$ & $3.55e-01$ & $2.00$ & $1.69e-02$ & $2.07$ & $3.08e-01$ & $1.89$\\
$2^6$ & $5.47e-04$ & $3.00$ & $8.90e-02$ & $2.00$ & $4.17e-03$ & $2.02$ & $8.71e-02$ & $1.80$\\
$2^7$ & $6.85e-05$ & $3.00$ & $2.22e-02$ & $2.00$ & $1.03e-03$ & $2.01$ & $2.62e-02$ & $1.73$\\
\hline
\end{tabular}
\end{table}

\end{example}


\section{Appendix A: Derivative transformation of curlinear mesh} 
\label{app:A}

Since the curved-edge quadrilateral element \(K^h\) is obtained from the reference element \(\widehat K\) through the mapping  $
\Psi_K(\xi,\eta)=(x(\xi,\eta),y(\xi,\eta))$.
we now derive the relation between the directional derivatives on \(K^h\) and those on \(\widehat K\). Let \(v\) be a differentiable function on \(K^h\), and let \(\widehat{v}_{K^h}=v\circ\Psi_K\). Denote 
$\nabla v= (\frac{\partial v}{\partial x},\frac{\partial v}{\partial y})^T$
and 
$
\widehat\nabla \widehat v_{K^h}= (\frac{\partial \widehat{v}_{K^h}}{\partial \xi}, 
\frac{\partial \widehat{v}_{K^h}}{\partial \eta})^T$.
By the chain rule, we have 
\begin{equation}\label{eq3} 
\widehat\nabla \widehat{v}_{K^h}=\mathbb {J}_K\nabla v, 
\end{equation} 
where 
\[
\mathbb {J}_K=D\Psi_K 
= 
\begin{pmatrix} 
x_\xi & y_\xi\\ 
x_\eta & y_\eta 
\end{pmatrix}. 
\] 
For any unit vector \(\bm n=(n_1,n_2)^T\), the directional derivative of \(v\) in the direction \(\bm n\) is 
\begin{align} 
\frac{\partial v}{\partial\bm n} 
&=\bm n\cdot\nabla v 
=\bm n\cdot \mathbb {J}^{-1}_K\widehat\nabla \widehat{v}_{K^h} \nonumber\\ 
&= 
\frac{1}{J_K} 
\left[ 
\left(y_\eta n_1-x_\eta n_2\right) 
\frac{\partial\widehat{v}_{K^h}}{\partial\xi} 
+ 
\left(-y_\xi n_1+x_\xi n_2\right) 
\frac{\partial\widehat{v}_{K^h}}{\partial\eta} 
\right], 
\label{normal_derivative_general} 
\end{align} 
where \(J_K\) is the determinant of the Jacobian matrix \(\mathbb {J}_K\). 
 
We next consider two special normal directions. For fixed \(\eta\), the curve 
\(\Psi_K([-1,1],\eta)\) has tangent vector \((x_\xi,y_\xi)^T\), we choose the unit normal vector as 
$\bm n_\xi=\frac{1}{r_\xi}(-y_\xi,x_\xi)^T$
where $r_\xi=\sqrt{x_\xi^2+y_\xi^2}$.
For fixed \(\xi\), the curve \(\Psi_K(\xi,[-1,1])\) has tangent vector \((x_\eta,y_\eta)^T\), we choose the unit normal vector as $\bm n_\eta=\frac{1}{r_\eta}(y_\eta,-x_\eta)^T$
where $r_\eta=\sqrt{x_\eta^2+y_\eta^2}.$ Define $s=x_\xi x_\eta+y_\xi y_\eta.$
Substituting \(\bm n_\xi\) into \eqref{normal_derivative_general}, we obtain, across the curve \(\Psi_K([-1,1],\eta)\), 
\begin{equation}\label{eq4} 
\frac{\partial v}{\partial\bm n_\xi} 
= 
\frac{1}{r_\xi J_K} 
\left( 
-s\frac{\partial\widehat v_{K^h}}{\partial\xi} 
+ 
r_\xi^2\frac{\partial\widehat v_{K^h}}{\partial\eta} 
\right). 
\end{equation} 
Similarly, substituting \(\bm n_\eta\) into \eqref{normal_derivative_general}, we obtain, across the curve \(\Psi_K(\xi,[-1,1])\), 
\begin{equation}\label{eq5} 
\frac{\partial v}{\partial\bm n_\eta} 
= 
\frac{1}{r_\eta J_K} 
\left( 
r_\eta^2\frac{\partial\widehat v_{K^h}}{\partial\xi} 
- 
s\frac{\partial\widehat v_{K^h}}{\partial\eta} 
\right). 
\end{equation}

\section{Conclusion}
In this paper, we develop a high-order finite-volume method on a curved-edge quadrilateral mesh based on the Gaussian dual partition. Under given assumptions regarding mesh regularity and geometric approximation, we prove the stability of this method and derive \(H^1\) and \(L^2\) error estimates that explicitly account for geometric errors. Finally, three sets of numerical experiments validate the theoretical results and demonstrate the effectiveness of this method.

\textbf{Acknowledgments}
This work is supported by the National Key Laboratory of Computational Physics (Grant No. SYSM-2006-WDZC-13) and the National Natural Science Foundation
of China (Grant No. 12671493).


\end{document}